\documentclass[11pt]{amsart}
\usepackage{amssymb}
\usepackage{amsmath}
\usepackage{amsthm}
\usepackage{wrapfig, subfigure,graphicx}
\usepackage[usenames]{color}
\usepackage{amsmath,amsfonts,latexsym,graphicx,amssymb,url}
\usepackage{amsmath,txfonts,pifont,bbding,pxfonts,manfnt}
\usepackage{psfrag}
\usepackage{hyperref}
\usepackage[active]{srcltx}
\usepackage{pdfsync}
\usepackage[none]{hyphenat}
\newcommand{\dist}{\text{dist}} 
\newcommand{\diam}{\text{diam}}

\newcommand{\supp}{\text{{\rm supp }}}

\newcommand{\R}{{\mathbb R}} 

\def \e {\varepsilon}

\newtheorem{theorem}{Theorem}[section]
\newtheorem{corollary}{Corollary}[section]
\newtheorem{lemma}{Lemma}[section]
\newtheorem{proposition}{Proposition}[section]
\newtheorem{definition}{Definition}[section]
\newtheorem{remark}{Remark}[section]

\numberwithin{equation}{section}

\def\di{\hbox{div}\,}
\def\Lip{\hbox{Lip}\,}

\def\bR{\mathbb{R}}
\def\d{\partial}

\def\ep{\epsilon}
\def\de{\delta}

\def\Om{\Omega}

\def\vfi{\varphi}

\def\cE{\mathcal E}
\def\cH{{\mathcal{H}}}
\def\cI{{\mathcal{I}}}
\def\cJ{{\mathcal{J}}}
\def\cL{{\mathcal{L}}}

\def\d{\partial}
\def\proof{{\it Proof.}\quad}

\begin{document}

\title{Special cases of the planar least gradient problem}
\author{Wojciech G\'orny, Piotr Rybka, Ahmad Sabra}\address{Faculty of Mathematics, Informatics and  Mechanics\\
The University of Warsaw\\
ul. Banacha 2, 02-097 Warsaw, POLAND}

\maketitle

\begin{abstract}
We study two special cases of the planar least gradient problem. In the first one, 
the boundary conditions are imposed on a part of the strictly convex domain. In the second 
case, we impose the Dirichlet data on the boundary of a rectangle, an example of convex but not strictly convex domain. We show the existence of solutions and study their properties for particular cases of data. 
\end{abstract}

\section{Introduction}\label{Si}
We are interested in the two-dimensional case of the least gradient problem 
\begin{equation}\label{ri1}
 \min\left\{\int_\Omega | D u|:\ u \in BV(\Omega),\ T_{\Gamma}u= f\right\},
\end{equation}
where $\Gamma$ is an open subset of $\d\Omega$. Here, $T_\Gamma u$ denotes the restriction of the trace $Tu$ to $\Gamma$.

We establish connections between \eqref{ri1} and the following problem appearing in the Free Material Design (FMD), see \cite{lewinski}, \cite{kocvara},
\begin{equation}\label{ri2}
 \inf\left\{\int_\Omega | p |:\ p \in L^1(\Omega;\bR^2),\ \di p = 0, \ p\cdot \nu|_\Gamma = g\right\},
\end{equation}
where $\nu$ is the outer normal to $\partial \Omega$, and $p\cdot \nu|_\Gamma$ is a properly defined trace operator. Since $u$ appearing in \eqref{ri1} is a scalar function as opposed to a vector field $p$ in \eqref{ri2}, we may say that \eqref{ri1} represents a dimensional reduction of \eqref{ri2}.

We note at this stage that a regularity assumption on $\partial \Omega$ must be made so that the trace is well-defined, see \cite{frid}.

In  Section \ref{Spe}, we show how to reduce \eqref{ri2} to \eqref{ri1} when $\Omega$ is convex. We discuss there the various definitions of traces and we show how the boundary data $f$ and $g$ are
related.

Once this is done, we address the following problems:\\
1) Boundary data may be specified on a set $\Gamma$ essentially smaller than $\d\Omega$, the rest of the boundary remains free. We will study this problem in Section \ref{sec:PartialBoundary} when $\Omega$ is strictly convex with $C^1$ boundary and $\Gamma$ is an arc. Our existence result, Theorem \ref{main}, is shown for continuous $f$ having one-sided limits at $\Gamma$ endpoints. Uniqueness and further properties are studied in Theorem \ref{main2}.\\
2)  Loads at the boundaries in (\ref{ri2}) may be concentrated at a few points and be zero elsewhere. This corresponds to piecewise constant $f$ in (\ref{ri1}). This is analyzed in Section \ref{s:e} also for strictly convex $C^1$ domains $\Omega$.\\
3) We relax the strict convexity condition in Section \ref{s:re}, and consider the basic case of $\Omega$ being a rectangle. In this case the data $f$ are continuous on $\d\Omega$. The main result is Theorem \ref{t:r}, which is applied to a special case in subsection \ref{sec:bFMD}.

The opening Section \ref{Spe} is devoted to  establishing a relationship between the two minimization problems. On the way, we have to discuss various notions of traces, which is particularly important when we want to compute the normal trace of  vector valued measures, see Definition \ref{def:normaltrace} and Remark \ref{rmk:measuretrace}. The basic properties of sets of finite perimeter are recalled in paragraph \ref{sub:per}.
 
The least gradient problem, \eqref{ri1}, with $\Gamma = \partial\Omega$, attracted attention of many authors, \cite{bomba}, \cite{nachman}, \cite{mazon}, \cite{mazon15}, \cite{miranda}, \cite{sternberg}, \cite{tamasan}, \cite{dosantos} who addressed \eqref{ri1}, as well as its anisotropic or non-homogeneous generalizations. 

We are particularly interested in the result by Sternberg {\it et al.}, see \cite{sternberg}, who showed existence and uniqueness of solutions of \eqref{ri1} when $\Gamma=\Omega$ and the boundary data $f$ are continuous. The authors of \cite{sternberg}
showed that solutions to the least gradient problem are less  smooth than the boundary data. 

The method of proof in \cite{sternberg} is constructive, and it is based on the fundamental result in \cite{bomba} showing that superlevel sets of solution to \eqref{ri1} are minimal surfaces.
In Section \ref{sec:PartialBoundary}, we extend the method  used in  \cite{sternberg}, to find solutions to \eqref{ri1} when $\Gamma \subsetneq \d\Omega$ by constructing explicitly their level sets, see paragraph \ref{sec:SWZ}.  We use this approach in Theorem \ref{main2} to study examples of boundary conditions for which we can  establish basic properties of solutions like uniqueness, continuity and presence of level set with positive Lebesgue measure. As opposed to the case when $\Gamma=\d \Omega$, solutions to \eqref{ri1} with continuous boundary data are not necessarily continuous in $\bar \Omega$. 

As we mentioned above, traces are an issue. Here, we will point to one of the aspects. It is well-known that the trace space of $BV(\Omega)$ is $L^1(\d\Omega)$. It is also well-known, see \cite{mazon}, that the functional defined in (\ref{ri1}) is not  lower semicontinuous on $\{u\in BV(\Omega):\ Tu = f\}$. This implies that the direct methods of the calculus of variations need not be best suited to solve \eqref{ri1}.
In \cite{mazon}, by a different method, the authors showed that a solution to 
\eqref{ri1} 
exists, provided that $f\in L^1(\partial\Omega)$. 
However, these solutions 
need not satisfy the boundary condition in the trace sense. This problem is highlighted in \cite{tamasan}, \cite{dosantos}, where the authors showed there that the space of traces of solutions to \eqref{ri1} is  smaller than $L^1(\d\Omega)$.  The second remark related to the construction in \cite{mazon} is the lack of uniqueness of solutions, even when $f$ has jump discontinuities in just a few points. These remarks show that we have to be very careful about the following:\\
a) solutions to \eqref{ri1} need not exist for the data in the space of traces of $BV$ functions;\\
b) uniqueness of solutions might be lost while relaxing continuity of the data.




In Section \ref{s:re}, we address the issue of  existence of solutions to \eqref{ri1}, when $\Gamma = \d\Omega$ and $\Omega$ is a rectangle, i.e. an example of  convex but not strictly convex set. For such sets the general theory in \cite{sternberg} fails. Our method of proof is based on approximation of rectangle $\Omega$ by strictly convex sets and on a stability result showed in \cite{miranda}. In Theorem \ref{t:r}, we show existence result provided that the boundary data are monotone on the sides of the set $\Omega$. 

\section{A  relationship between two minimization problems}\label{Spe}
In this section, we show relationship between solutions to the following problems,
\begin{equation}\label{eq:LG}
\min \left\{\int_{\Omega}|Du|:u\in BV(\Omega), T
u=f\right\}
\end{equation}
and
\begin{equation}\label{eq:FMD}
\inf \left\{\int_{\Omega}|p|: p\in L^1\left(\Omega;\R^2\right), \text{div } p=0,\, p\cdot \nu|_{\d\Omega} =g\right\}.
\end{equation}
The fist one,
\eqref{eq:LG}, is the well-known Least Gradient problem. For $\Omega$ with Lipschitz boundary, and $f\in L^1(\partial \Omega)$, the restriction $Tu$ at the boundary is understood in the sense of trace of $BV$ functions. 

The second problem, \eqref{eq:FMD}, appears in the Free Material Design, where the goal is to find the optimal material distribution of a body to support a load applied to its boundary, $p$ is a field in $L^1$ and  $\di p$ is understood in the distribution sense, see definition \ref{def:divergence}. For the normal trace $p\cdot \nu|_{\d\Omega}$ to be well defined, with $\nu$ the outer unit normal to $\partial \Omega$, the boundary of $\Omega$ should belong to a special class of Lipschitz domains called deformable Lipschitz see \cite[Definition 3.1]{frid2}. 
This class of sets  contains
convex domains, \cite[Remark 2.2]{frid}, which are studied in this paper.

We define $p\cdot \nu|_{\d\Omega}$ so that the Gauss-Green theorem is satisfied, see Definition \ref{def:normaltrace}. 
We introduce now the definition of the divergence of a field in $L^p$ , $1\leq p\leq \infty$, along with the notion of its normal trace. These definitions generalize to the case when the field is a Radon measure. 

\begin{definition}\label{def:divergence}\rm
A vector field $F\in L^{p}\left(\Omega;\R^n\right)$, $1\leq p\leq \infty$, or $F$ a vector valued Radon measure on $\Omega\subseteq \bR^n$ is called a {\it divergence measure field} if the distributional divergence $\di F$ is a Radon measure and its  total variation,
$$|\di F|(\Omega)=\sup \left\{\int_\Omega F\cdot \nabla \vfi:\vfi\in C_0^1(\Omega), 
\,||\vfi||_{L^\infty(\Omega)}\leq 1\right\},$$
is finite. 
\end{definition}

\begin{definition}\label{def:normaltrace}\rm
If $F\in L^p\left(\Omega;\R^n\right)$, $1\leq p\leq \infty$, is a divergence measure field, then the {\it normal trace}
$F\cdot \nu|_{\partial \Omega}$ is a continuous functional on $\Lip(\d \Omega)$, the space of Lipschitz continuous functions, defined
by the following formula
\begin{equation}\label{eq:Gauss}
\left \langle F \cdot \nu|_{\d\Omega},\varphi\right \rangle= \left \langle \text{div }F, \phi \right \rangle+\int_\Omega F(x)\cdot \nabla\phi (x)\,dx,
\end{equation}
where $\phi$ is the Whitney extension of $\varphi$ to $\Lip(\R^n)$ preserving the Lipschitz constant, see \cite[Section 2.10.43]{federer}.
By \cite[Theorem 3.1]{frid2}, the operator $F\cdot \nu|_\Omega$ is well defined as a continuous linear
functional on the space of Lipschitz continuous functions on $\d\Omega$ and it is independent of the choice of the extension $\phi$.
\end{definition}

\begin{remark}\label{rmk:measuretrace}\rm
In the case when $F$ is a vector valued measure whose distributional divergence is a measure, the normal trace, defined by \eqref{eq:Gauss}, is a continuous functional over the space of $\Lip(\gamma,\partial \Omega)$ with $\gamma>1$, \cite[Theorem 3.1]{frid2}. The space $\Lip(\gamma,C)$ can be understood as the space of Lipschitz function with bounded and continuous partial derivative of order not greater than $\gamma$, see \cite[Chapter VI, Section 2.3]{Stein} for a detailed description of these functions and their extensions.
\end{remark}

Since space $L^1$ is not weakly$^*$ closed, we do not expect the existence of minimizers of  \eqref{eq:FMD}. However, in this section, we establish a relation between the infimum of \eqref{eq:FMD} and solutions to \eqref{eq:LG}. This is described in the following theorem.

\begin{theorem}\label{thm:LGvsFMD} Let us assume $\Omega\subseteq \R^2$ is convex and $u$ is a solution to \eqref{eq:LG} with $Tu =f\in L^1(\d\Omega)$, then $g=\frac{\partial f}{\partial \tau}$ is a continuous functional over $\Lip(\d\Omega)$, and $q=R_{-\frac{\pi}{2}}D u$ is a solution to \eqref{eq:FMD} in the following sense 
\[ \inf \left\{\int_{\Omega}|p|: p\in L^1\left(\Omega;\R^2\right), \text{div } p=0,\, p\cdot \nu|_{\d\Omega} =g \text{ on }\partial \Omega\right\} = |q| (\Omega) \equiv |D u| (\Omega).
\]
Here, $R_\alpha$ denotes a rotation with angle $\alpha$ around the origin, and $\tau$ is the tangent such that $(\nu,\tau)$ is positively oriented.
Moreover, $q$ is a divergence measure field, with $\di q=0$, and $q\cdot\nu|_{\d\Omega}=g$ as defined in Remark \ref{rmk:measuretrace}.
\end{theorem}


In order to prove this theorem, we need the following result.

%
%

\begin{proposition}\label{pr1}
Suppose $p\in L^1(\Omega; \bR^2)$, 
 $\di p =0$ in the sense of distributions and $\Omega$ is convex, then there exists $u$, an element of $W^{1,1}(\Omega)$ such that $p=R_{-\frac{\pi}{2}} \nabla u$, more precisely,
  $$
 \int_\Omega |p| \,dx = \int_\Omega |\nabla u | \,dx.
 $$
 Moreover, if $p\cdot\nu|_{\d\Omega} = g$ in the sense of Definition \ref{def:normaltrace} and $Tu =f$, then $g=\frac{\d f}{\d\tau}$.
\end{proposition}
\proof We fix $x_0\in \Omega$, and define the differential form $\omega = p_1 dx_2 - p_2 dx_1$, where $p=(p_1,p_2)$. We set $p_\ep = p\ast\varphi_\ep:=(p_1^\ep,p_2^\ep)$, $\ep>0$, with $\varphi_\varepsilon$ a sequence of mollifiers. We define a sequence of mollified differential forms
$\omega^\ep = p^\ep_1 dx_2 - p^\ep_2 dx_1$ and the functions
$$
u^\ep(x) = \int_{\gamma_x} \omega^\ep .
$$
with $\gamma_x$ a line segment joining $x_0$ to $x$. Since $p_\varepsilon\to p$ in $L^1(\Omega;\R^2)$, then we shall see that $u^\ep$ converges in $W^{1,1}(\Omega)$.
In fact,
\begin{eqnarray*}
\int_\Omega |u^\ep(x)-u^\delta(x)|dx=\int_\Omega\int_0^{ |x-x_0|}
\left| (p^\ep(\gamma_x(t)) - p^\delta(\gamma_x(t))) \cdot \eta \right| |\dot\gamma| \,dtdx :=\cJ,
\end{eqnarray*}
where $\eta$ is a unit normal to the interval $\gamma_x$. We take $\sigma$ to be tangent to $\gamma_x$. We switch to another orthogonal coordinate system, such that the first axis is parallel to $\sigma$, while the second one is parallel to $\eta$. Then, by Fubini Theorem, we change the order of integration and we obtain,
\begin{eqnarray*}
&\cJ &\le \int_{\bR} \int_{ \bR }\int_{ \bR}
\left| 
(p^\ep(\gamma_x(t)) - p^\delta(\gamma_x(t)))\right| \chi_\Omega(\sigma x_1 +\eta x_2)\,dt dx_2dx_1\\
&&\le \diam(\Omega) \| p^\ep - p^\delta\|_{L^1(\Omega)}.
\end{eqnarray*}
We conclude that $u^\epsilon$ is a Cauchy sequence in $L^1$.

We
notice that, 
$\nabla u_\ep$ is a Cauchy sequence in $L^1$ too. Indeed, since $\di p=0$, then $\di p^\ep=0$, and $\nabla u^\ep=R_{\frac\pi2}p^\ep$, therefore
$$
\| \nabla u_\ep - \nabla u_\de \|_{L^1(\Omega)} \le \| R_{\frac\pi2}(p^\ep - p^\de) \|_{L^1(\Omega)} 
= \| p^\ep - p^\de \|_{L^1(\Omega)} .
$$
We denote by $u$ the $W^{1,1}$-limit of $u^\ep$. Since $p^\ep\to p$ in $L^1$, then $\nabla u=R_{\frac\pi2}p$ and 
$$
\int_\Omega |\nabla u| \,dx = \int_\Omega |R_{\frac{\pi}{2}}p| \,dx=\int_\Omega |p| \,dx
$$
Finally, we study the traces. If $\phi \in \Lip(\R^2)$, and $\vfi$ is its restriction on $\d\Omega$, then
\begin{eqnarray*}
\langle g,\vfi\rangle &= & \langle p\cdot \nu, \vfi \rangle = \int_\Omega p \cdot \nabla\phi \,dx = 
\lim_{\ep\to 0} \int_\Omega p^\epsilon \cdot \nabla\phi \,dx =
\lim_{\ep\to 0} \int_\Omega R_{-\frac{\pi}{2}} \nabla u^\epsilon\cdot \nabla\phi \,dx\\
&= & \lim_{\ep\to 0} \int_{\d\Omega} R_{-\frac{\pi}{2}} \nabla u^\epsilon\cdot \nu\, \vfi\,d\cH^1.
\end{eqnarray*}
Due to smoothness of $u^\epsilon$ we have 
$R_{-\frac{\pi}{2}} \nabla u^\epsilon\cdot \nu = \frac{\d u^\epsilon}{\d \tau}$. Hence,
\begin{equation*}
 \langle g,\vfi\rangle = \lim_{\ep\to 0} \int_{\d\Omega} R_{-\frac{\pi}{2}} \nabla u^\epsilon\cdot \nu \,\vfi\,d\cH^1
 = \lim_{\ep\to 0} \int_{\d\Omega} \frac{\d u^\epsilon}{\d \tau} \vfi\,d\cH^1 =
 -\lim_{\ep\to 0} \int_{\d\Omega} u^\epsilon \frac{\d\vfi}{\d\tau}\,d\cH^1. 
\end{equation*}
Since $u^\epsilon$ converges to $u$ in $W^{1,1}$, then the traces converge in $L^1(\d\Omega)$, $T u^\epsilon \to Tu$.
Thus,
\begin{equation*}
 \langle g,\vfi\rangle =- \int_{\d\Omega} Tu \frac{\d\vfi}{\d\tau}\,\,d\cH^1 =- \int_{\d\Omega} f \dfrac{\d\vfi}{\d\tau}\,\,d\cH^1
 =\langle\dfrac{\d f}{\d\tau}, \vfi\rangle . \eqno\Box
\end{equation*}



\begin{remark}\rm The above proof fails when $p$ is a measure since  the sequence of smooth functions $u^\ep$ will not, in general, converge in $BV$.  
\end{remark}

\bigskip \textit{\textbf{Proof of Theorem \ref{thm:LGvsFMD}.}} Let us suppose that
$u$ is a solution to \eqref{eq:LG} and $P$ is the value of the infimum of \eqref{eq:FMD}. 
We have to show that $q:= R_{-\frac{\pi}{2}} D u$ satisfies $\int_\Omega |q| = P$ and $q\cdot \nu|_{\d\Omega} =\frac{\partial f}{\partial \tau}=g$.

Let us suppose that $p_n\in L^1(\Omega)$ is a minimizing
sequence, i.e. 
$$
\int_\Omega |p_n|\,dx \to P,\qquad \di p_n=0,\qquad p_n\cdot\nu|_{\d\Omega} =g.
$$
By Proposition \ref{pr1}, we can find a sequence $v_n\in W^{1,1}(\Omega)$, such that 
$$\int_\Omega |p_n|\,dx =\int_\Omega |\nabla v_n|\,dx.
$$ 
Moreover, we may require $Tv_n =f$.

We conclude that
$$
P=\lim_{n\to \infty} \int_\Omega |p_n|=\lim_{n\to \infty} \int_\Omega |\nabla v_n|\geq \int_\Omega |D u|=\int_\Omega |q|.
$$
Since $p_n$ is a minimizing sequence we infer that $P$ cannot be bigger than $\int_\Omega |D u|$.

Notice that $\di q=0$. In order to show that $q$ satisfies the desired boundary conditions it suffices to verify identity \eqref{eq:Gauss}. We take a sequence of smooth functions $w_n$ in $BV(\Omega)$ converging to $u$ in a strict sense, i.e., $w_n \to u$ in $L^1(\Omega)$ and $\int_\Omega |\nabla w_n|\,dx \to \int_\Omega |D u|$ and we require that $Tw_n = f= Tu$. Then, there is a subsequence (not relabeled) such that $\nabla w_n$ converges to $D u$ weakly as measures. Hence, if we set $\hat p_n = R_{-\frac{\pi}{2}}\nabla w_n$, then by Proposition \ref{pr1}, we have
$$
g_n =\hat  p_n\cdot\nu|_{\d\Omega} = \frac{\d T w_n}{\d\tau}.
$$
We know that $T w_n = f$, then we have
$$
g_n =  \frac{\d f}{\d\tau}.
$$
Since $\nabla w_n  \overset{\ast}{\rightharpoonup} Du$, then $\hat p_n  \stackrel{\ast}{\rightharpoonup} q$. 
If $\phi\in\Lip(\gamma, \R^2)$ with $\gamma>1$, then $\phi$ has continuous partial derivatives of order $1$ and hence,
$$
\int_\Omega \nabla \phi\,dq = \lim_{n\to\infty} \int_\Omega \nabla \,\phi \hat p_n\,dx.
$$
Subsequently, by \eqref{eq:Gauss} and $\di q=0$, we have
\begin{equation}\label{r:gr}
 \langle q\cdot\nu|_{\d\Omega}, \vfi\rangle = \int_\Omega \nabla\phi \,dq = \lim_{n\to\infty} \int_\Omega \nabla \phi \,\hat p_n\,dx = \lim_{n\to\infty}
 \langle \hat p_n\cdot \nu|_{\d\Omega},\varphi\rangle = 
 \langle g, \varphi\rangle ,
\end{equation}
where $\varphi$ is the restriction of $\phi$ to $\partial \Omega$.
Hence, $q$ satisfies the desired boundary conditions. 
\qed

\subsection{Case with load on $\Gamma\subsetneq \d \Omega$} In this section, we consider the case when the load is applied only on one part of the boundary. In other words, we assume that $\Gamma\subsetneq \Omega$ is an open set, and consider the following two minimizing problems,
\begin{equation}\label{eq:partialLG}
\min\left\{\int_\Omega |Du|:u\in BV(\Omega), T_\Gamma u =f\right\}
\end{equation}
and
\begin{equation}\label{eq:partialFMD}
\inf\left\{\int_\Omega |p|:p\in L^1(\Omega, \bR^2),\di p=0, p\cdot \nu|_\Gamma=g\right\}.
\end{equation}
We have to define the trace $p\cdot \nu$ on $\Gamma$. It is convenient to consider the general framework of Definitions \ref{def:divergence} and \ref{def:normaltrace}. We notice that $\Lip_\Gamma(\d\Omega) =\{ \varphi\in \Lip(\d\Omega):\ \varphi|_{\d\Omega \setminus \Gamma} = 0\}$ is a closed subspace of $\Lip(\d\Omega)$. We then define $p\cdot \nu|_\Gamma$ as a restriction of $p\cdot \nu|_{\d\Omega}$ to $\Lip_\Gamma(\d \Omega).$
Also, $T_\Gamma u$ is the restriction of $Tu$ on the set $\Gamma$.

Therefore, similarly as in Theorem \ref{thm:LGvsFMD} we obtain, the following result,
\begin{theorem}\label{thm:PartialLGvsFMD}
Assume $\Omega\subseteq \R^2$ is convex, $\Gamma\subsetneq \Omega$ is relatively open, and $u$ is a solution to \eqref{eq:partialLG} with $f\in L^1(\d\Omega)$, then $q=R_{-\frac{\pi}{2}}D u$ satisfies
$$
\inf\left\{\int_\Omega |p|:p\in L^1(\Omega),\di p=0, p\cdot \nu|_\Gamma=g\right\} = \int_\Omega \,|q| =
\int_\Omega |Du|,
$$
where $g=\frac{\partial f}{\partial \tau}$. Moreover,  $\di q=0$, $q\cdot \nu|_{\Gamma} =g$ as functional on $\Lip_{\Gamma}(\partial \Omega,\gamma)$,
with $\gamma>1$.
\end{theorem}

\section{Analysis of the least gradient problem, when $\Gamma\subsetneq\partial\Omega$}\label{sec:PartialBoundary}
In this section, we study solutions of the least gradient problem, when the trace is defined only on a part of the boundary. More precisely, we prove existence of minimizers of \eqref{eq:partialLG} when $\Gamma$ is open and $f$  is continuous in $ \Gamma$. 
Before embarking on the analysis of \eqref{eq:partialLG}, we point out the fact that in the case when the trace is defined over the whole boundary, the question of existence of solutions was studied in the literature from different perspectives. One could try to apply the direct method of the Calculus of Variation to the relaxed functional $\int_\Om |Du|$ to establish existence of minimizers, \cite{mazon15}. However, in this paper, we adapt the method introduced by \cite{sternberg} since it gives use insights about the structure of the solutions. This method requires modification and even with continuous data, uniqueness and continuity of solution might fail. 

\subsection{Sets of finite perimeter}\label{sub:per}
Before showing the existence of minimizers of \eqref{eq:partialLG}, 
we recall notions, frequently used in this paper, related to sets of finite perimeter and to the generalized definition of its boundary.
\begin{definition}\label{defi:finiteperimeter}
For $E\subseteq \Omega$ measurable, we say $E$ is a set of finite perimeter in $\Omega$ if and only if $\chi_E\in BV(\Omega)$. In this case, we define the perimeter of $P(E,\Omega)$ of the set $E$ as the total variation of the Radon measure $D\chi_E$.
\end{definition}

\begin{definition}\label{defi:boundaries}
Suppose $E$ is a set of finite perimeter.
\begin{itemize}
\item For $x\in \R^n$, the measure theoretic exterior normal $\nu(x,E)$ at $x$ is a unit vector $\nu$ such that
$$
\lim_{r\to 0} \dfrac{|B(x,r)\cap\{y:(y-x)\cdot \nu <0, y\notin E\}|}{r^n}=0,
$$
and
$$
\lim_{r\to 0} \dfrac{|B(x,r)\cap\{y:(y-x)\cdot \nu >0, y\in E\}|}{r^n}=0.
$$
(We restrict our attention to the case $n=2$.) 
\item The reduced boundary $\partial^* E$ is the set of points $x$ such that $\nu(x,E)$ exists.
\item The measure theoretic boundary $\partial_M E$ is the set of points $x\in R^n$ such that
$$\limsup_{r\to 0} \dfrac{|E\cap B(x,r)|}{|B(x,r)|}>0\qquad \text{and}\qquad \liminf_{r\to 0} \dfrac{|E\cap B(x,r)|}{|B(x,r)|}<1.$$
\end{itemize}
We have that $\partial^*E\subseteq \partial_M E\subseteq \partial E$, where $\partial E$ is a topological boundary of $E$. It is a well-known fact that $\mathcal H^1(\partial_M E\setminus \partial ^* E)=0$, and that
\begin{equation}\label{eq:important}
P(E,\Omega)=\mathcal H^1(\Omega \cap \partial_M E)=\mathcal H^1(\Omega\cap \partial^* E).
\end{equation}
\end{definition}

Since sets of finite perimeter are defined up to measure zero, then in order to avoid ambiguity, we will use the following convention,
\begin{equation}\label{e:des}
 x\in E \iff \limsup_{r\to 0} \dfrac{|E\cap B(x,r)|}{|B(x,r)|}>0.
\end{equation}

\subsection{Sternberg-Williams-Ziemer construction}\label{sec:SWZ}
We present here a version of a construction presented in \cite{sternberg}. We adapt it to deal with the case of the boundary condition on $\Gamma$. The justification of the method is based on the observation made in \cite{bomba}, which is recalled in Lemma \ref{bombieri} below. Before we state it, we recall the definition of functions of least gradients.

\begin{definition}\rm 
We say $u\in BV(\Omega)$ is a {\it function of least gradient} if 
$$\int_\Omega|Du|\leq \int_\Omega |D(u+w)|$$
for every compactly supported function $w\in BV(\Omega)$ with $\supp w\subseteq \Omega$.
\end{definition}

\begin{lemma}\label{bombieri}
 If $u$ is a solution to \eqref{eq:partialLG},
 then $\partial\{ u\ge t\}$ is a minimal surface for every $t\in \bR$.
\end{lemma}
\proof
Indeed, for any $v$ in $BV$ with $T_{\Gamma}v=f$ we have
\begin{equation}\label{eq:LGinequality}
\int_\Omega |Du| \le\int_\Omega |Dv|.
\end{equation}
Let $K\subseteq\Omega$ compact and $w\in BV(\Omega)$ with $\supp \,w=K$, then 
$
u|_{\Gamma} =(u+w)|_{\Gamma}=f.
$
Then, applying \eqref{eq:LGinequality} to $v=u+w$ we get that
$$
\int_\Omega |Du|\leq \int_\Omega |D(u+w)|.
$$
Thus, we have shown that $u$ is a function of least gradient, we can then apply the results by \cite{bomba} to conclude minimality of  $\partial\{ u\ge t\}$. \qed

Now, we are ready to state of the main results of this paper:

\begin{theorem}\label{main}
 Let us suppose that $\Omega$ is a strictly convex region in the plane with Lipschitz continuous boundary.
 $\Gamma\subsetneq\d\Omega$ is an open subset of $\d\Omega$ and $\Upsilon= \d\Omega\setminus\Gamma$ is a smooth arc with endpoints  $a$ and $b$. We assume that $f:\Gamma\to \bR$ is a bounded continuous function having one-sided limits at $a$ and $b$, denoted by $f(a)$ and $f(b)$, respectively. Then, there exists $u\in C(\Omega\cup\Gamma)$, a solution to \eqref{eq:partialLG}, i.e.,
\[
  \min\left\{ \int_\Om|D u |: \ u\in BV(\Omega), T_{\Gamma}u =f\right\},
\]
 where $T_\Gamma u =f$  is interpreted in the sense of the trace of BV functions.
\end{theorem}

First, we provide the setup for the existence proof. We can take any 
region $\Omega_o$, with Lipschitz boundary, containing $\Omega$ and such that 
$$
\d \Omega_o \cap \d\Omega = \Upsilon.
$$

Let $F\in C(\Omega_o) \cap BV(\Omega_o\setminus\bar\Omega)$ be an extension of $f$ to $\Omega_o $. 
Let  $ \cI=\overline {f(\Gamma)}$, recall that $\Gamma$ is an arc, so $\cI$ is a compact interval. For every $t\in \cI$, we define the set 
$$
\cL_t =\left\{ x\in \Omega_o: F(x)\geq t \right\}.
$$
Since $F\in BV(\Omega_o\setminus \bar \Omega)$, then $P\left(\cL_t,\Omega_o \setminus \bar \Omega\right)<\infty$ for almost every $t$. We let 
$$T=\cI\cap \{t:P(\cL_t,\Omega_o\setminus \bar \Omega)<\infty\}.$$
For every $t\in T$, we consider the minimization problem, 
\begin{equation}\label{3l}
 \min\left \{ P(E,\Omega_o):\  E\setminus\bar\Omega = \cL_t\setminus\bar\Omega \right\}.
\end{equation}
By compactness of the embedding $BV(\Omega_o)\subset L^1(\Omega_o)$, and the lower semicontinuity of the $BV$ norms imply the existence of minimizers of \eqref{3l}.

Among minimizers, we select one which is a solution to
\begin{equation}\label{3d}
 \max\left\{ |E|:\ E\hbox{ is a solution to \eqref{3l}}\right\}.
\end{equation}
The existence of a maximizer follows from the same argument as above.
It is easy to see that there is a unique maximizer up to measure $0$. In fact, assume $E^1$, $ E^2$ are maximizers of \eqref{3d}, notice first that $E^1\cup E^2\setminus \bar \Omega=\mathcal L_t\setminus \bar \Omega$, and similarly $E^1\cap E^2\setminus \bar \Omega=\mathcal L_t\setminus \bar \Omega$, then $E^1\cup E^2$ and $E^1\cap E^2$ are competitors to $E^1$ and $E^2$ in \eqref{3l}.
We know that for $A$ and $B$ of finite perimeters, we have
\begin{equation}\label{eq:unionintersection}
P(A\cup B, \Omega_o)+P(A\cap B ,\Omega_o)\leq P(A,\Omega_o)+P(B, \Omega_o)
\end{equation}
Therefore, $E^1 \cup E^2$ and $E^1 \cap E^2$ are minimizers of \eqref{3l}.
Since $E^1,$ and $E^2$ maximize of problem \eqref{3d}, then
\begin{align*}
|E^1|&\geq |E^1\cup E^2|=|E^1|+|E^2\setminus E^1|,\\
|E^2|&\geq |E^1\cup E^2|= |E^2|+|E^1\setminus E^2|.
\end{align*}
We deduce that $|E^1\bigtriangleup E^2| =0$, concluding the proof of the claim.

For every $t$ in $T$, we denote the maximizer of \eqref{3d} by $E_t$. 

The argument in the proof of \cite[Lemma 3.3]{sternberg} is local in its nature, so we obtain similarly the following key property of $E_t$:

\begin{lemma}\label{le2}
 If $\Gamma\subsetneq\d\Omega$ is open and $f: \Gamma\to\bR$ is continuous and $E_t$ is given by \eqref{3d}, then $\d E_t\cap \Gamma \subseteq f^{-1}(t)$ for every $t\in T$. \qed
\end{lemma}

An important step in the construction is checking the ordering of $E_t$ expected for candidates for the level sets.  First, we need the following geometrical result, that relies on the fact that $\d\Omega\setminus\Gamma$ is smooth and $\Omega$ is strictly convex:
\begin{lemma}\label{l:ort}
If $\Upsilon$ is a smooth arc, $\gamma$ is a connected component of
$\d E_t$ in $\Omega$ 
intersecting $\Upsilon^\circ$, the
interior of $\Upsilon$, then $\gamma$ is orthogonal to $\Upsilon$. 
\end{lemma}
\proof
Let $x_t\in \partial E_t\cap \Upsilon^\circ$.  We know that $\d E_t$ is a minimal surface in $\R^2$, then $\gamma$ must be a segment $[x_t,y_t]\subseteq \partial E_t$ intersecting $\Upsilon^\circ$ at $x_t$. We will show that $[x_t,y_t]$ is normal to $\partial \Omega$ at $x_t$. Take a ball  $B(x_t, \epsilon)$, such that $B(x_t, \epsilon)\cap \Gamma=\emptyset$, and let $p_t\in \Omega$ the intersection of $\d B(x_t,\varepsilon)$ and $[x_t,y_t]$. Suppose now, that our claim does
not hold, then convexity of $B(x_t, \epsilon)\cap \Omega$ and the Pythagorean Theorem imply that
dist$(p_t,\Upsilon)< $ dist$(x_t,p_t)$. This implies that we can make $P(E_t, \Omega_o)$ smaller, contrary to our minimality assumption.\qed



\begin{lemma}\label{leo}
 Let $s,t\in T$ be such that $s<t$. If 
 $$
 \d E_t \cap \partial \Gamma= \emptyset\qquad
 \hbox{or}\qquad \d E_s \cap\partial \Gamma = \emptyset,
 $$
 then $E_t \Subset E_s$.
\end{lemma}
\begin{proof}
We prove first that $E_t\subseteq E_s$. In fact, since $\cL_t\subseteq \cL_s$, then 
\begin{equation*}
 E_s\cup E_t \setminus \bar \Omega = \cL_s \setminus \bar \Omega,\qquad
E_s\cap E_t \setminus \bar \Omega = \cL_t \setminus \bar \Omega  .
\end{equation*}
Therefore by \eqref{3l}, $P(E_s\cup E_t, \Omega_o)\geq P(E_s, \Omega_o)$ and $P(E_s\cap E_t, \Omega_o)\geq P(E_t, \Omega_o)$. Now, \eqref{eq:unionintersection} implies that  $P(E_s\cup E_t, \Omega_o)= P(E_s, \Omega_o)$. As a result, \eqref{3d} yields $|E_t \cup E_s| \le |E_s|$
i.e.
$|E_t\setminus E_s|=0$, hence due to (\ref{e:des}), $E_t\subseteq E_s$.


We will show that $\partial E_s\cap \partial E_t\cap( \partial \Omega\setminus \d\Gamma)=\emptyset$.
In fact, we have 
by Lemma \ref{le2}, $\partial E_s\cap \partial E_t\cap \Gamma =\emptyset$, so it is enough to prove that 
$\partial E_s\cap \partial E_t\cap \Upsilon^\circ=\emptyset$. Let us assume otherwise and take
$x_0\in \d  E_t \cap \d E_s \cap \Upsilon^\circ$. By Lemma \ref{l:ort}, the connected components of  $\d E_t$ and $\d E_s$, which intersect $\Upsilon$ must both 
meet $\Upsilon$ orthogonally. Since  $\Upsilon$ is smooth, then $\d  E_t$ and $\d E_s$ must be contained in the line perpendicular to $\Upsilon$ and passing through $x_0$, hence they
must coincide. As a result,  $\cL_s=\cL_t$, a contradiction. 


It remains to show that $\partial E_t\cap \partial E_s\cap \Omega=\emptyset$. Indeed, 
since $\d E_t$ and $\d E_s$ are minimal surfaces in $\R^2$, then each of their components in $\Omega$ is a segment. Hence, if they intersect in $\Omega$, then by \cite[Theorem 2.2]{sternberg} they must coincide. Thus, we reached a contradiction.
\qed
\end{proof}

We set $A_t := \overline{E_t \cap \Omega}$, for every $t\in T$. 
Following \cite{sternberg}, we define our candidate for a solution to \eqref{eq:partialLG}, by the formula,
\begin{equation}\label{3u}
u(x) = \max\{ t\in T: \ x\in A_t\}. 
\end{equation}
We have to make sure that for each $x\in \Omega$, there is $t\in T$ such that  $x\in A_t$. Indeed, if $t=\inf_{\Gamma} f -\epsilon$, $\epsilon>0$, then  we have that $P(\cL_t,\Omega_o\setminus \bar \Omega)=0$ and $A_t\cap \Omega = \Omega$. Similarly, if $t=\sup_{\Gamma} f +\epsilon$, then  $P(\cL_t,\Omega_o\setminus \bar \Omega)=0$.
 

\begin{lemma}\label{lem:inclusion} For every $t\in T$, we have
$$\{x\in \Gamma: f(x)>t\}\subseteq E_t^\circ\cap \Gamma\subseteq A_t\cap \Gamma \subseteq \{x\in \Gamma: f(x)\geq t\}.$$
\end{lemma}

\begin{proof}
In order to prove the first inclusion we
take $x\in \Gamma$ such that $f(x)>t$.  Then, there exists a neighborhood of $x$ in $\Omega_o\setminus \Omega$ such that $F(x)>t$, therefore by the definition of $E_t$, we deduce that $x\in \bar E_t$. By Lemma \ref{le2}, $x\notin \partial E_t$, hence $x\in E_t^\circ$.

Now, we prove the second inclusion. Let $z\in E_t^\circ\cap \Gamma$, there exists a neighborhood $V_z$ of $z$ such that $V_z\subseteq E_t$. Since $z\in \Gamma$, then there exists a sequence $z_n\in E_t\cap \Omega$ such that $z_n\to z$, thus $z\in A_t$.
%
%

To prove the third inclusion, notice that $A_t\cap \Gamma\subseteq \bar E_t\cap \Gamma$. Take $x\in \bar E_t \cap \Gamma=(\partial E_t \cap \Gamma) \cup (E_t^{\circ}\cap \Gamma)$. By Lemma \ref{le2}, it is enough to show that $E_t^{\circ}\cap \Gamma \subseteq \{x\in \Gamma: f(x)\geq t\}$. Assume $x\in E_t^\circ\cap \Gamma$, then there exists a neighborhood of $x$, $V_x\subseteq E_t^\circ$ and by the definition of $E_t$, $F(y)\geq t$ for all $y\in V_x\cap (\Omega_o\setminus \bar \Omega)$, and hence by continuity $f(x)\geq t$.\qed
\end{proof}

The following Lemma follows directly from Lemma \ref{leo}, and the definition of $A_t$.
\begin{lemma}\label{lem:inclusionofAt}
If  $\d E_t \cap \partial \Gamma= \emptyset\,$
or $\,\d E_s \cap \partial \Gamma = \emptyset,
 $
 then $A_t\Subset A_s$ for $s<t$.
\end{lemma}

We summarize the main achievement of the construction.
\begin{proposition}\label{p:co}
 Let us suppose that $u$ is given by the above version of the SWZ construction, then $u$ is continuous in $\Omega\cup \Gamma$
 and $u =f$ on $\Gamma$.
\end{proposition}

\begin{proof} We notice first that
$$\{x:u(x)\geq t\}=\bigcap_{s\in T, s<t}A_s\,\quad \hbox{and}\quad 
\{x:u(x)>t\}=\bigcup_{ s\in T, s>t} A_s.$$
The first set is obviously closed. We will prove that $\bigcup_{ s\in T, s>t} A_s\cap \Omega$ is open. Indeed,  let us take $x_0\in \bigcup_{ s\in T, s>t} A_s\cap \Omega$, so there is $s_0>t$ such that $x_0\in A_{s_0}$ i.e. $u(x_0) \ge s_0>t$. Moreover, $\dist(x_0,\d\Omega) >0$. We take
\begin{equation*}
 r = \frac12 \min\{ \dist(x_0,\d\Omega), \dist(x_0,\d A_t)\}>0.
\end{equation*}
Of course $B(x_0,r)\subset \Omega$.
We claim that $B(x_0,r)\subset \bigcup_{ s\in T, s>t} A_s$. For any $x \in B(x_0,r)$, we have $\dist(x,\d A_t)>0$, so 
$u(x)>t$ and the claim follows.
Hence, the continuity of $u$ in $\Omega$.

We shall prove now that for every $x_0\in \Gamma$, $\lim_{x\to x_0, x\in \Omega}u(x)=f(x_0)$.
We define $t=f(x_0)$ and take any $s$,  $s<t$. Then, $f(x_0)>s$ and by Lemma \ref{lem:inclusion}, $x_0\in E_s^\circ\cap \Gamma$. Hence,  there exists a neighborhood $V_{x_0}$ such that $V_{x_0}\cap \Omega\subseteq E_s\cap \Omega \subseteq A_s$, therefore $u(z)\geq s$ for all $z\in V_{x_0}\cap \Omega\subseteq A_s$. Thus, we deduce that $\liminf_{z\to x_0}u(z)\geq s$ for all $s<t$, hence $\liminf_{z\to x_0}u(z)\geq t$. It remains to show that $\limsup_{z\to x_0}u(z)\leq t$. Let us assume otherwise and take $a=\limsup_{z\to x_0}u(z)> t$. Let $r\in (t,a)$, then there exists a sequence $z_n\in \Omega$ so that $u(z_n)\geq  r$. Hence $z_n\in A_r\cap \Gamma$, since $A_r$ is closed, we deduce that $x_0\in A_r\cap \Gamma$. Thus, by Lemma \ref{lem:inclusion}, we deduce $f(x_0)\ge r>t$, thus we reached a contradiction. 
\qed

\end{proof}
%

\textbf{\textit{Proof of Theorem \ref{main}}.} It remains to show that $u$ is a function of least gradient. Let $v$ be a competitor of $u$, i.e. $v\in BV(\Omega)$ and $T|_\Gamma v=f$. We extend $v$ onto $\Omega_o$ and we call the extension by $\tilde v$, we require also that $\tilde v=F$ on $\Omega_o\setminus \Omega$. Notice that $\tilde v \in BV(
\Omega_0)\cap C(\Omega_o\setminus \bar \Omega)$. Let $t\in \cI$, for $\mathcal H^1$-almost every $x\in \Gamma$,

$$\lim_{r\to 0}\fint_{B(x,r)\cap \Omega} |\tilde v(y)-f(x)|\,dy=0.$$
Let $F_t=\{v\geq t\}$, and $\partial^*F_t$ be the reduced boundary of $F_t$. We shall show that for every $x\in \partial^*F_t\cap \Gamma$, $f(x)=t$. Let us assume otherwise, i.e. $f(x)=t-\varepsilon$, then
\begin{align*}
0&=\lim_{r\to 0}\dfrac{1}{|B(x,r)\cap \Omega|}\left(\int_{B(x,r)\cap \Omega\cap\{\tilde v<t\}}|\tilde v(y)-f(x)|\,dy + \int_{B(x,r)\cap \Omega\cap\{\tilde v\geq t\}}|\tilde v(y)-f(x)|\,dy\right)\\
&\geq \limsup_{r\to 0}\dfrac{1}{|B(x,r)\cap \Omega|}\int_{B(x,r)\cap \Omega\cap\{\tilde v\geq t\}}|\tilde v(y)-f(x)|\,dy\\
&\geq \varepsilon \limsup_{r\to 0}\dfrac{|B(x,r)\cap \Omega\cap F_t|}{|B(x,r)\cap \Omega|}.
\end{align*}
Using the fact that $f$ is the trace of $\tilde v$ as a function on $\Omega_o\setminus \bar\Omega$, we obtain similarly that
$$ \limsup_{r\to 0}\dfrac{|B(x,r)\cap (\Omega_o\setminus \Omega)\cap F_t|}{|B(x,r)\cap \Omega|}=0.$$
Hence,
$$\lim_{r\to 0} \dfrac{|B(x,r)\cap \Omega_o\cap F_t|}{|B(x,r)|}=0.$$
This contradicts the fact that $\partial^*F_t\subseteq \partial_M F_t$. 

We come to a similar conclusion if we assume $f(x)>t$, hence $f(x)=t$. Therefore, $\mathcal H^1(\partial^* F_t\cap \Gamma)\leq \mathcal H^1(f^{-1}(t))=0$ for all but countable many $t$.
Thus, by \eqref{eq:important}
\begin{align}
P(F_t,\Omega_o)&=\mathcal H^1(\Omega_o\cap \partial^* F_t)=\mathcal H^1(\Omega \cap \partial^*F_t)+\mathcal H^1(\Gamma \cap \partial^*F_t)+\mathcal H^1(\Omega_o\setminus \bar \Omega\cap \partial^*F_t)\notag\\
&=P(F_t,\Omega)+\mathcal H^1(\partial^* \cL_t\setminus \bar \Omega).\label{eq:result1}
\end{align}


Similarly, by Lemma \ref{le2}, $\mathcal H^1(\partial^*E_t\cap \Gamma)=0$, then
\begin{align}
P(E_t,\Omega_o)&=\mathcal H^1(\Omega_o\cap \partial^* E_t)=\mathcal H^1(\Omega \cap \partial^*E_t)+\mathcal H^1(\Gamma \cap \partial^*E_t)+\mathcal H^1(\Omega_o\setminus \bar \Omega\cap \partial^*E_t)\notag\\
&=P(E_t,\Omega)+\mathcal H^1(\partial^* \cL_t\setminus \bar \Omega).\label{eq:result2}
\end{align}
Since $F_t\setminus \bar \Omega=\mathcal L_t\setminus \bar \Omega$, then by \eqref{3l}
\begin{equation}\label{eq:result3}
P(E_t,\Omega_o)\leq P(F_t,\Omega_o).
\end{equation}

Therefore \eqref{eq:result1},\,\eqref{eq:result2}, and \eqref{eq:result3} yield $P(E_t,\Omega)\leq P(F_t,\Omega)$ which together with the coarea formula conclude the proof of the theorem.
\qed

\subsection{Uniqueness of solution to (\ref{eq:partialLG})}\label{Sm}

In this section, we study uniqueness  of solutions to \eqref{eq:partialLG}, constructed in Theorem \ref{main}. For this purpose we keep assumptions of this theorem, i.e. strict convexity of $\Omega$ and continuity of $f$. Our basic  tool
is based on the observation that a solution is fully specified by its level sets and their `labeling', see Lemma \ref{p:8}  and Lemma \ref{flag}.  


Our first observation is concerned with the range of solutions to \eqref{eq:partialLG}. Recall that $\mathcal I=\overline{f(\Gamma)},$ we set $M=\sup_{\Gamma} f$, $m=\inf_{\Gamma} f$.

\begin{proposition}\label{p:range}
Let $u$ be a solution to \eqref{eq:partialLG}, then the range of $u$ is contained in $\mathcal I$, where the inclusion is up to measure $0$. \end{proposition}
 \proof If we set
 $v = \min\{M,\max\{m,u\}\}$, then $v\in BV(\Omega)$, and $v$ satisfies 
$$
\int_\Omega |Dv| \le \int_\Omega |Du|.
$$
Notice that the range of $v$ is contained in $\cI$. If $|\{ u>M\}\cup\{ u<m\}|>0$, then the inequality above is strict, which is a contradiction, because $T_\Gamma v=T_\Gamma u=f$. \qed

\medskip
We are interested in detecting fat level sets of solutions to \eqref{eq:partialLG}.

\begin{definition}
Let $u$ be a solution to \eqref{eq:partialLG}, and $\mathcal E_t=\{u\geq t\}$. For $t\in u(\Omega)$, we say that the $t$-level set of $u$ is {\it fat}  if
$$
\left|\bigcap_{s<t} \mathcal E_s \setminus \bigcup_{s>t} \mathcal E_s\right| >0.
$$
\end{definition}

\bigskip
We would like to gain insights into the solutions constructed in Section \ref{sec:SWZ}. For this purpose, we denote by $u_0$ the solution obtained by means of the Sternberg-Williams-Ziemer construction in Theorem \ref{main}, and $\mathcal E_t^0=\{x:u_0(x)\geq t\}$. Our goal is to give an explicit description of the sets $\d \mathcal E_t^0$ for all $t\in u_0(\Omega)$, when $f$ is continuous. Moreover, these sets will be determined only in terms of $f$, $\Omega$ and $\Gamma$. 
In other words, we will show that if $f$ satisfy a monotonicity condition, and $u$ is another  solution to \eqref{eq:partialLG}, $\partial\mathcal  E_t=\{u\geq t\}$, then  
\begin{equation}\label{r:row}
 \partial \mathcal E_t^0= \partial \mathcal E_t\qquad\hbox{for all}\ t.
\end{equation}
An important step toward a uniqueness proof is showing that (\ref{r:row}) implies
\begin{equation}\label{r:rsy}
 |\mathcal E_t^0\bigtriangleup\mathcal E_t | = 0,\qquad\hbox{for all }t\in \bR,
\end{equation} see Lemma \ref{p:8}.
From Lemma  \ref{flag} we will then deduce that 
$u_0=u$ a.e.

This plan is carried out in Theorem \ref{main2} for a few examples of continuous functions $f$, but we expect in future research to establish uniqueness results for a larger class of functions. Before we state this theorem, we introduce more notation aiming at capturing
the structure of solutions to the least gradient problems.
We will describe the geometric aspects of the SWZ construction solutions 
in terms of solutions to 
\begin{equation}\label{3di}
 d(x,\Upsilon) =  d(x,y),\qquad y\in\Upsilon,
\end{equation}
when $x\in\Gamma$ is given. We notice that since $\Upsilon$ is closed, then there exists a solution to (\ref{3di}), but it needs not to be unique.  
We also define
\begin{equation}\label{eq:defS}
 S:= \{x\in\Gamma:\hbox{ there exists }y\in \Upsilon^\circ
 \hbox{ such that }d(x,y) = d(x,\Upsilon)\},
\end{equation}
where $\Upsilon^\circ$ denotes the relative interior of $\Upsilon$. We also split $S$, 
\begin{equation}\label{eq:defU}
U= \{x\in S:\ \hbox{there is a unique solution to (\ref{3di})}\}
\end{equation}
and
\begin{equation}\label{eq:def}
D = \{x\in S:\ \hbox{there are at least two solutions to (\ref{3di})}\}.
\end{equation}

In order to study the structure of the set of solutions to(\ref{3di}), we 
define a map $\Phi: S\to P(\Upsilon)$  by
$$
\Phi(x) = \{ y\in \Upsilon:\ d(x,y)=d(x,\Upsilon)\}.
$$
\begin{proposition}\label{p:cc}
 Let $C$ be a connected component of $D$, then $C$ is a singleton, and hence $D$ is at most a countable set.
\end{proposition}
\proof Assume otherwise, i.e. suppose there exist  $p$ and $q$, distinct points in $C$. 
Since $C$ is a one dimensional connected set, then the arc $\overline{pq}$ must  be contained in $C$.

We take $x_1, x_2\in \Phi(p)$ and $y_1, y_2\in \Phi(q)$ such that 
\begin{equation}\label{r:g}
 d(x_1, x_2) = \hbox{diam}\,\Phi(p),  \qquad d(y_1, y_2) = \hbox{diam}\,\Phi(q).
\end{equation}
We claim that the arcs $\overline{x_1x_2}$ and $\overline{y_1y_2}$  may not intersect. Let us suppose they do, then 
we have two possibilities:\\
1) the intersection point is in the relative interior of the arcs $\overline{x_1x_2}$ and $\overline{y_1y_2}$;\\
2) the intersection point is one endpoint, say $x_2$.

In the first case, one endpoint of  $\overline{x_1x_2}$, say  $x_2$ is in the interior of $\bigtriangleup_{qy_1y_2}$, the curvilinear triangle with vertices $q$, $y_1$ and $y_2$,  where $y_1$ and  $y_2$ are connected by the circle arc centered at $q$. Notice that, since $d(q,\Upsilon^\circ)=d(q,y_1)=d(q,y_2)$, then the disc of center $q$ passing through $y_1,$ and $y_2$ contains $\bigtriangleup_{qy_1y_2}$. This implies that 
$$
d(q,x_2) < d(q, y_1) = d(q, y_2),
$$
contrary to the definition of $y_1$ and  $y_2$.

In the second case, we consider any point $s$ belonging to the arc $\overline{pq}$, different from $p$ and $q$. By our assumption $s\in C$, hence $\Phi(s)$ contains at least two different points, $z_1$ and $z_2$ and
$$
d(s,z_1) = d(s,z_2),
$$
and the arc  $\overline{z_1z_2}$ must intersect the interior of at least one curvilinear triangles, $\bigtriangleup_{px_1x_2}$ or $\bigtriangleup_{qy_1y_2}$. Let us suppose that $z_2$ is in the interior of
$\bigtriangleup_{qy_1y_2}$. Then,
$$
d(q,z_2) < d(q, y_1) = d(q, y_2).
$$
This is again contrary to the definition of $y_1$ and  $y_2$.
  

  
Thus, we proved that if $p,q\in C$, then the corresponding arcs $\overline{x_1x_2}$ and $\overline{y_1y_2}$, which have positive length, do not intersect. Since the number of elements of $C$ is not countable, we infer that the total length of the corresponding arcs is infinite, thus we reached a contradiction. \qed
 



\bigskip
Now, we are ready to state the main result of this section.
\begin{theorem}\label{main2}  Let us suppose that the assumptions of Theorem \ref{main} hold and $u_0$ is a solution given by Theorem \ref{main}.\\
 1) If $f(a) = f(b) = \inf_{\d\Omega}f$, $f$ has a strict maximum at $x_M$ and $f$  attains each value (except $\max_{\d\Omega} f$) exactly twice,
 then $u_0$ is discontinuous, unique 
 and there is $\tau\in(\inf_{\d\Omega} f, \max_{\d\Omega} f)$ such that the $\tau$-level set is fat.\\
 2) If $f$ is monotone, 
 then $u_0$ is discontinuous. 
 There is at least one $\tau$, an element of $u(\Phi(D))$, 
 such that the $\tau$-level set is fat. Moreover, $u_0$ is a unique solution in the class of solutions continuous in $\Omega$.\\
 3) Let us suppose that $f(a) = f(b)$ and  
 there are $x_m \in\Gamma$ a unique local minimum, $x_M \in\Gamma$ a unique local maximum, and $x_0\in D\subseteq\Gamma$ such that
 $f(x_0) = f(a) = f(b)$,
 \begin{equation}\label{A}
  d(x_0,\Upsilon) =  d(x_0,a) =  d(x_0,b).
 \end{equation}
Moreover, we assume that $f$ restricted to arcs $\overline
{ax_m}$, $\overline
{x_mx_M}$, $\overline
{x_Mb}$ is one-to-one.
Then, $u_0$ is continuous, unique and the $f(a)$-level set is fat.
\end{theorem}
\begin{remark}\rm 
 We can change `min' to `max' in this theorem to get a corresponding result. We notice that  function $f$ on an arc may be increasing or decreasing depending on a parametrization of arc $\Gamma$. The same claim in 3) holds if we require
$\Phi(D) \cap \{a,b\}=\emptyset.$

One might expect that in case 1), the least gradient function should take value $\inf_{\d\Omega}f$ on $\Upsilon$, but it turns out that forming discontinuities decreases the energy $\int_\Omega|Du|$.
\end{remark}

\bigskip
We want to present here  different types of behavior of $f$ near the endpoints of $\Upsilon$. We obtain quite detailed information about solutions $u_0$ due to the explicit construction of $\d \mathcal E_t^0$, what enables us to study the level sets, $\mathcal E_t^0=\{ u_0(x) \ge t \}$.

The main claim is uniqueness of solutions. Detecting fat level sets is a by-product. We proceed in two stages. Firstly,
we discover the structure of $\d \mathcal E_t^0$, which depends only on the geometry of $\d\Omega$ and properties of $f$. Then, we show that we can reconstruct $u_0$ and any other solution from this information. 

Before we study the special cases in Theorem \ref{main2}, we present a more general uniqueness argument, which will be applicable also in Section \ref{s:re}. 

\begin{lemma}\label{p:8} Let us assume that $u_0$ and $u$ are solutions to \eqref{eq:partialLG}, where $u_0$ is constructed in Theorem \ref{main}, i.e. $u_0$ is continuous on $\Omega \cup \Gamma$. Recall $\mathcal E_t = \{ x\in \Omega: \ u(x)\ge t\}$, $\mathcal E_t ^0= \{ x\in \Omega: \ u_0(x)\ge t\}$.
If $\partial \mathcal E_t =\partial \mathcal E_t^0$ for all $t$, then $|\mathcal E_t \bigtriangleup \mathcal E_t^0|= 0$ for all $t\in\bR$.
\end{lemma}
\proof {\it Step 1.}
Let us first assume that $u_0$ has no fat level sets, so each $x\in\Omega$ belongs to $\partial \mathcal E_t^0$ for an appropriate, unique $t$.  We take any $x$, which is a density point of $\mathcal E_t^0$. Thus, $s=u_0(x)>t$ and due to continuity of $u_0$ there exists such a ball $B(x,r)$, that for all points $y\in B(x,r)$ we have 
$u_0(y)> \frac{s+t}2$. Subsequently, we notice
\begin{equation}\label{r:bs}
B(x,r) = \bigcup_{\tau > \frac{s+t}2} B(x,r)\cap \partial \mathcal E_\tau^0
= \bigcup_{\tau > \frac{s+t}2} B(x,r)\cap \partial \mathcal E_\tau.
\end{equation}
This is so, because each $x\in\Omega$ belongs to some $\d \mathcal E_\tau^0$ and
$\partial \mathcal E_t=\partial \mathcal E_t^0$ holds for all $t$. Formula (\ref{r:bs}) also implies that $x$ is the density point of $\mathcal E_t$.
This argument shows that
$\mathcal E_t^0 \subseteq \mathcal E_t$ up to a measure zero set. We shall see that
for all $t$, the measure $\mathcal E_t \setminus \mathcal E_t^0$ must be zero. Let us suppose otherwise, so there is $t_0$ such that $|\mathcal E_{t_0} \setminus \mathcal E_{t_0}^0| > 0.$ If $x$ is in $V:=\mathcal E_{t_0} \setminus \mathcal E_{t_0}^0$, then $u(x) \ge t_0$ while $u_0(x)<t_0$. We notice that
$$
V= \bigcup_{s<t_0} V \cap \partial \mathcal E_s^0,
$$
due to lack of fat level sets through each point of $\Omega$, exactly one segment $\partial \mathcal E_s^0$ passes.
As a result, 
we have
$$
V= \bigcup_{s<t_0} V \cap \partial \mathcal E_s.
$$
Since $V$ has a positive measure, the above statement implies that no point in $V$ is a density point of $\mathcal E_{t_0}$. We reached a
 contradiction, our claim follows. 

{\it Step 2.} Now, we consider the case, when there is a fat $t_0$-level set. Let us suppose that 
$x$ is  a density point of $\mathcal E_{t_0}^0$. There are two possibilities:

(a) There is $r>0$ such that $u_0$ restricted to $B(x,r)$ equals $t_0$. Then, by assumption for all $s\neq t_0$, we have $\d \mathcal E_s^0\cap B(x,r)=\emptyset.$ Thus, for all $s\neq t_0$, we have $\d \mathcal E_s\cap B(x,r)=\emptyset.$ So, $t_0$-level set of $u$ is fat too.

(b) For all $r>0$, we have 
$$
0< |\{u_0>t_0\} \cap B(x,r)| / |B(x,r)| <1.
$$
This implies existence of a sequence $x_n \to x$ and $u_0(x_n) = s_n \to t_0$. By continuity of $u_0$, we have that $u_0> \frac{s_n+t_0}{2}$ in a ball $B(x_n,\rho_n)$. We conclude $\cE^0_{t_0} \subset \cE_{t_0}$.

Now, we have to show that $\cE_{t_0} \setminus \cE^0_{t_0}=\emptyset$. If this does not happen, then there is $x_0\in \cE_{t_0} \setminus \cE^0_{t_0}$, this implies that $x_0\in\cE^0_\tau$ for a $\tau <t$. Then, we consider the cases:
1) $x_0\in\d\cE^0_\tau$,
2) $\cE^0_\tau$ is fa level set.

Using the argument of the first part of the proof, we get a contradiction with $u(x_0)\ge t$.
We conclude that $\cE^0_{t_0} = \cE_{t_0}$.

The argument above tells us that the fat level sets of $u_0$ and $u$ are the same.
\qed

\begin{lemma}\label{flag}
 Let us suppose that $v,w\in BV(\Om)$. Then, $v=w$ a.e. if and only if
 $$
 |\{ v\ge t\} \bigtriangleup\{ w\ge t\} | = 0,\qquad\hbox{for all }t\in \bR,
 $$
 where $A\bigtriangleup B$ denotes the symmetric difference of sets $A$ and $B$.
\end{lemma}
\proof If $v=w$ a.e. then the superlevel sets must coincide, up to a negligible set.

Conversely, let us assume that $A:= \{ v\ge t\} \setminus\{ w\ge t\}$ has a positive measure. Then, for all $x\in A$, we have
$$
v(x)\ge t, \qquad\hbox{but}\qquad w(x)<t.
$$
Thus, $v$ and $w$ differ on a set of positive measure. The same argument shows that $\{ v\ge t\} \setminus\{ u\ge t\}=0.$
\qed

\medskip

We would like to establish an analogue of Lemma \ref{le2}, with $f$ and $\Omega$ as in the statement of Theorem \ref{main},
for solutions, which are not necessarily continuous.
If 
$u$ is such a solution to \eqref{eq:partialLG}, 
then Proposition \ref{p:range} implies that $u(x) \in[\inf_\Gamma f,\sup_\Gamma f]$ for a.e. $x$. 

Let us set $\mathcal E_t=\{u\ge t\}$. We  know that $\d\mathcal E_t$ are 
minimal surfaces. They must intersect $\d\Omega$, i.e. $\d \mathcal E_t\cap\d\Omega\neq\emptyset$. A given $\d \mathcal E_t$ is a sum of intervals and none of their endpoints may belong to $\Omega$. If any of them did, e.g.
$x^t\in \d \mathcal E_t$ is in $\Omega$, then at this point 
$\mathcal E_t$ has zero density. But this is impossible due to our convention \eqref{e:des}.

\begin{lemma}\label{l:st}
If $u$ is any solution to \eqref{eq:partialLG}, 
then
$$
\d \mathcal E_t\cap \Gamma \subseteq f^{-1}(t).
$$
\end{lemma}
\proof
Indeed, if $x_0\in \d \mathcal E_t\cap \Gamma$, then
$$
f(x_0) = \lim_{r\to 0} \fint_{B(x_0,r)\cap\Omega} u(y)\,dy.
$$
If $f(x_0)<t$ were true, then there would exist $\rho>0$ such for all $x\in B(x_0,\rho)\cap \Gamma$ we have $f(x)<t$. Hence, we could deduce from existence of the trace that function $u$ assumes values smaller than $t$ in
$\mathcal E_t\cap B(x_0,r)$  for all $r>0$, but this
is impossible. The same argument makes $f(x_0)> t$ impossible too. As a result $f(x_0)= t$. 
\qed

\bigskip Introducing a parametrization of $\Gamma$ is advantageous for further considerations.
If we use the convention that $[0,L]\ni s\mapsto x(s)$ is an arc length parametrization of $\Gamma$ and $x(0) =a$, $x(L)=b$, then we define the following sets,
$$
B_a = \{ s\in(0,L): 
d(x(s),\Upsilon) = d(x(s),a)\},\qquad
B_b = \{ x\in(0,L): 
d(x(s),\Upsilon) = d(x(s),b)\}.
$$
We  shall write $\tilde S = x^{-1}(S)$, where $S$ is defined in \eqref{eq:defS}
. 
Let us set $s_a = \sup 
B_a$, $s_b =\inf 
B_b$. We notice that $s_a\in 
B_a$, and  $s_b \in 
B_b$. 
Indeed, let us take $s_n\in 
B_a$ converging to $s_a$, then after introducing the shorthand $x_n = x(s_n)$,
$$
d(x(s_a),a) = \lim_{n\to\infty} d(x_n,a) = \lim_{n\to\infty} d(x_n,\Upsilon) = d(x(s_a),\Upsilon).
$$
It is obvious from the definition that
\begin{equation}\label{r:a}
 s_a \le \inf \tilde S, \qquad \sup  \tilde S \le s_b.
\end{equation}
We prove a more precise statement, namely equalities hold
in \eqref{r:a}. 
\begin{corollary} Let us suppose that $S$, defined in (\ref{eq:defS}), is not empty, then 
 using the notation introduced above, we have that
$$
\sup B_a =\max B_a  \in B_a,\quad\inf B_b =\min B_b  \in B_b,
\quad \max B_a  = \inf \tilde S,\quad \min B_b =\sup \tilde S. 
$$
\end{corollary}
\proof
It is sufficient to consider one 
of inequalities in (\ref{r:a}), because the argument is the same in both cases.
If $s_a<s_0<\inf\tilde S$, then $d(x(s_0),\Upsilon) <d(x(s_0),a)$. At the same time $d(x(s_0),\Upsilon)$ may not be attained at any $y\in Y^\circ$.
As a result, we deduce that $d(x(s_0),\Upsilon) = d(x(s_0), b$), but this is impossible if 
$S \neq \emptyset$. \qed 


\bigskip
\noindent{\bf Remark.} Let us stress that due to a possible non-uniqueness of solutions to (\ref{3di}), it may also happen that $x(\sup 
B_a) \in S$. 

\bigskip
{\bf \it Proof of Theorem \ref{main2}, Part 1).}
We assume $f(a) = f(b) = \inf_{\d\Omega} f$ and 
$u$ is  any  solution to (\ref{eq:partialLG})
$\mathcal E_t =\{ u \ge t\}$. Due to Lemma \ref{l:st} we know that $\d\cE_t\cap \Gamma\subset f^{-1}(t)$. We set
\begin{equation}\label{d:tu}
\mathcal T_\Upsilon = \{ t\in \cI:\ \d \mathcal E_t\cap \Upsilon \neq\emptyset\},\qquad
\mathcal T_\Gamma = \{ t\in \cI:\ \d \mathcal E_t\cap \Upsilon= \emptyset\}.
\end{equation}
Since each value of $f$, except for $\inf_{\d\Omega} f$ and $\max_{\d\Omega} f$ is attained exactly twice, then
we introduce 
$$
f^{-1}(t) =\{x^t,y^t\}\qquad \hbox{for }t\in (\inf_{\d\Omega} f, \max_{\d\Omega} f) .
$$

In order to prove that $\mathcal T_\Gamma$ and $\mathcal T_\Upsilon$ are non-empty, we take $t$ such that $t-f(a)>0$ is so small that $x^t \in x(B_a)$ and $y^t \in x(B_b)$ and
$$
d(x^t, y^t) > d(x^t,a) +d(b, y^t).
$$
Hence, $\d \cE_t$ must consist of the segments $[x^t, a]$ and $[b,y^t]$, i.e. $\mathcal T_\Upsilon\neq\emptyset$.

Take $x_M$, the strict, global maximum of $f$. Then, $\Gamma \setminus\{x_M\}=\Gamma_1\cup\Gamma_2$ and on each arc $\Gamma_i$, $i=1,2$, the function $f$ is monotone. Let  $t$ be such that $ f(x_M) - t$ is small. Thus, $x^t$ and $y^t$ are close to $x_M$ and 
$$
d(x^t, y^t) <d(x^t,a) +d(b, y^t).
$$
Hence, $\d\mathcal  E_t$ must be segment $[x^t, y^t]$, so $\mathcal T_\Gamma\neq\emptyset$.

We notice that if we  take $t\in T_\Upsilon $, $s\in \mathcal T_\Gamma$, then the only possibility is $s>t$.
Moreover, at least one of $\d \mathcal E_s$, $\d \mathcal E_t$ does not intersect $\d\Gamma$. Hence,
inequality $s<t$ implies that $\mathcal E_t \Subset \mathcal E_s$, which in turn implies that the  intersection of $\d\mathcal E_t$ and $\d \mathcal E_s$ is not empty, but  this is impossible.
As a result,
$$
\inf \mathcal T_\Upsilon \ge \sup \mathcal T_\Gamma .
$$
Since the strict inequality above is not possible, we may conclude:
 \begin{equation}\label{r:tau}
 \inf \mathcal T_\Upsilon = \sup \mathcal T_\Gamma =:\tau.
\end{equation} 

We would like to make sure that $\tau$ is in  $\mathcal T_\Upsilon$ and in $\mathcal T_\Gamma$. In order to show this we
recall that
$x^t\in\Gamma_1$ and $y^t\in\Gamma_2$. Moreover, the functions
$$
t\mapsto x^t,\ y^t
$$
are continuous, because of the strict monotonicity of $f|_{\Gamma_i}$, $i=1,2$. Thus, we infer continuity of
$$
d(x^t, a) + d(y^t,b) - d(x^t, y^t) =: h(t).
$$
Since $h(f(a))<0$ and $h(f(x_M))>0$, we deduce the existence of a point where $h$ vanishes.  Moreover, this implies that $\tau \in \mathcal T_\Gamma \cap \mathcal T_\Upsilon$ and $\tau$ is defined uniquely.

We claim that this $\tau$ is the number postulated in the statement of Part 1).
In order to check it, 
let us take  any $t>\tau$, where $\tau$ is defined in (\ref{r:tau}), then the assumptions of Lemma \ref{lem:inclusionofAt} are satisfied, hence
$\mathcal E_t \Subset \mathcal E_\tau$ 
and $\d \mathcal E_t$ intersects $\Gamma.$ 
If we take a sequence $s_n > \tau$ converging to $\tau$, then we can deduce that $x^{s_n}\to x^\tau$ and
$y^{s_n}\to y^\tau$ and $\mathcal E_{s_n}\Subset \mathcal E_\tau$. Therefore,
$$
\bigcup_{\epsilon>0} \mathcal E_{\tau +\epsilon} = \Omega \cap H^+(x^\tau,y^\tau),
$$
where $ H^+(x^\tau,y^\tau)$ is the closed half plane with  boundary containing $[x^\tau,y^\tau]$ and $\d  H^+(x^\tau,y^\tau) \cap  \d\Gamma = \emptyset$.

On the other hand, for $t<\tau$, we have $\d \mathcal E_t \cap\Upsilon \neq\emptyset$.  Thus, $\d\mathcal  E_t\neq [x^t,y^t]$. Since $t<\tau$, there are $p^t$, $q^t\in \Upsilon$ such that 
$$
\d \mathcal E_t = [x^t,p^t]\cup [y^t,q^t].
$$
Let us suppose that $t_n<\tau$ converges to $\tau$. Then,
$$
p^{t_n}\to p^\tau,\qquad q^{t_n}\to q^\tau
$$
and $p^\tau$, $q^\tau \in\Upsilon$.
As a result $\tau\in T_\Upsilon$. 
At the same time 
$$
\bigcap_{\epsilon>0} \mathcal E_{\tau -\epsilon} = \Omega \cap H^+(x^\tau,p^\tau)\cap H^+(x^\tau,q^\tau),
$$
where half planes $ H^+(x^\tau,p^\tau)$ and $H^+(x^\tau,q^\tau)$ are so chosen that their intersection does not contain neither $a$ nor $b$. We see that
$$
\{ u(x) =\tau \} =
\bigcup_{\epsilon>0} \mathcal E_{\tau +\epsilon} \setminus \bigcap_{\epsilon>0}\mathcal  E_{\tau -\epsilon}
=  \Omega \cap H^+(x^\tau,y^\tau)\setminus  \Omega \cap H^+(x^\tau,p^\tau)\cap H^+(x^\tau,q^\tau).
$$
This set has positive Lebesgue measure, i.e. the  $\tau$-level set is fat. It is also easy to see that there no other fat level sets.

Now, we establish uniqueness of solutions.
Indeed, if $u$ is any solution 
and
$t<\tau$, then $\d \mathcal E_t=[x^t, y^t]$, where
$x^t\in B_a$, $y^t\in B_b$. For $t>\tau$, then $t\in \mathcal T_\Gamma$ and again the structure of $\d \mathcal E_t$ is known. We can see that for each point $x$  in $\Omega\setminus\{u = \tau\}$ 
there exists a unique $\d \mathcal E_t$, such that $x\in \d \mathcal E_t$. Hence, we can apply
Lemma \ref{p:8} and then Lemma \ref{flag}. This ends the proof of Part 1).

\bigskip
{\bf\it Proof of part 3).} We proceed as in the proof of part 1). For a solution $u$, we set $\mathcal E_t=\{u\ge t\}$. Let us examine $\d  \mathcal E_t$ for $t\in(f(a),f(x_M))$. 
We know that $\d \mathcal E_t \cap \Gamma \subseteq f^{-1}(t)$. We have $f^{-1}(t)=\{x^1_t,x^2_t\}$.
Actually,
we claim that $\d \mathcal E_t \cap \Gamma = f^{-1}(t)$.  If it were otherwise, then $\d \mathcal E_t$ would have contained a point $y$ from $\Upsilon$. We notice that (\ref{A}) implies that $y = a$ or $y=b$. Otherwise $\d \mathcal E_t$ would intersect $\mathcal E_{f(a)}$ in $\Omega$, but this is not possible. Thus, $\d \mathcal E_t$ must consist of $[x^1_t,c]$ and $[x^2_t,c]$, where $c=a$ or $c=b$. But due to the triangle inequality $d(x^1_t,x^2_t)< d(x^1_t,c) +  d(x^2_t,c)$, contrary to minimality of $P(\mathcal E_t,\Omega_o)$. This means that the triangle with vertices $x_0$, $a$ and $b$ is contained in $\{u= f(a)\}$. Hence, the $f(a)$-level set is fat.
We also deduce continuity of $u$ at points belonging to $\Upsilon$.

Uniqueness follows by the same argument as in Part 1), thus the proof of Part 3) is complete.

\bigskip
{\it Proof of part 2)}\ Suppose $u_0$ is given by Theorem \ref{main} and $\cE_t^0$ are its superlevel sets. By Proposition \ref{bombieri}, $\d \mathcal E_t^0$ are minimal surfaces, i.e., intervals $\ell^t$.
They must intersect $\d\Omega$, so $\ell^t = [x^t,y]$, $y\in\Upsilon$.

Obviously, there is $x_0\in \Gamma$ such that $d(x_0,a)=d(x_0,b)$, hence $x_0$ is an element of $D$. 
Moreover, by Proposition \ref{p:cc} we know that $D$ is at most countable.
Momentarily we show that elements of $f(\Phi(D))$ have fat level sets. 

If $x_0\in D$ and $\Phi(x_0)\supset\{y_1,y_2\}$, then the set bounded by $[x_0,y_1]$, $[x_0,y_2]$ and the arc $\overline{y_1y_2}$, denoted by $\bigtriangleup_{x_0}$, has positive measure and $u|_{\bigtriangleup_{x_0}} = f(x_0)$. Next, we  check that for each $x$ in $\Omega\setminus \bigcup_{x_i\in D} \bigtriangleup_{x_i}$ there is $\cE^t$ 
such that $x\in \d\mathcal  E_t$. If it were otherwise, then we would have two cases to consider:
$$
1)\quad \min_{t\in\cI}\dist(x,\d \cE_t^0) >0\qquad\hbox{or}\qquad
2)\quad\inf_{t\in\cI}\dist(x,\d \cE_t^0) =0.
$$
The first case implies existence of a ball $B(x_0,r)$ such that no $[x^t,y^t]$ intersects it.  Then we would see that $B(x_0,r)$ is contained in a fat level set, but we are considering $\Omega$ with all fat level sets removed, so we reached a contradiction.

In the second case there would be a sequence $t_n$ converging to $t_0$ and such that $\dist (x,\d \cE^0_{t_n}) \to 0$. There is also a sequence $x_n$ converging to $x$ such that $u(x_n) = t_n$.
Since $u_0$ is constructed in Theorem \ref{main} is
continuous in $\Omega$, we deduce that $u_0(x) = t_0$.
Since for all $t_n \neq t_0$, then  we deduce that $x\in \d\cE^0_{t_0}$ contrary to the assumption.

We describe the level set structure with the help of $B_a$, $B_b$, $S$.
We  observe that $f^{-1}(t)$ consists of a single point $x^t$, 
$t\in(\inf_{\d\Omega}f,\sup_{\d\Omega}f)$. Since $\d\cE_t^0$ is a line segment $[x^t, y]$, then $y\in \Upsilon$.

Since we discovered the structure of the level sets of continuous, then uniqueness of solutions follows in this class of function due to the argument used in the earlier Parts. \qed

\section{Least gradient problem on rectangles}\label{s:re}
In this section, we consider the least gradient problem, when $\Omega$ is a rectangle. Notice that in this case, $\Omega$ does not satisfy the assumptions in \cite{sternberg}. We study the case when the data $f$ are continuous strictly monotone on two pieces of the boundary and prove existence of continuous least gradient solution. We next show uniqueness of these solution in the $BV$ setting. Notice that for general boundary data, least gradient solutions might not be continuous nor unique see \cite[Example 2.7]{mazon}.
We state the main result of this section in the following theorem:

\begin{theorem}\label{t:r}
We are given a rectangle $\Omega=(-L,L)\times(-h,h)$.
 We denote $v_i = \{(-1)^{1+i} L\}\times (-h,h)$ and
$h_i= (-L,L)\times \{(-1)^{1+i}h\}$ , $i=1,2$ the sides of $\Omega$. 
We write
 $\d\Omega= \Gamma_1 \cup \Gamma_2$, where $\Gamma_1=h_1\cup v_1$, $\Gamma_2=h_2\cup v_2$.
 We are given a continuous function $f$, such that $f|_{\Gamma_i}$ is
strictly  monotone with $i=1,2$,
Then there exists a unique continuous solution to 
 \begin{equation}\label{rr}
 \min\left\{\int_\Omega | D u|:\ u \in BV(\Omega),\ T_{\d\Omega}u= f\right\},
\end{equation}
where the boundary condition is understood in terms of the trace of $BV$ functions.
\end{theorem}

Before proving our theorem, 
we recall the following stability result, proved in \cite{miranda}.

\begin{proposition}\label{thm:Miranda}
If $u_n$ is a sequence in of least gradient functions converging in $L^1(\Omega)$ to  function $u$, then $u$ is a function of least gradient. \qed
\end{proposition}

\subsection{Proof of Theorem \ref{t:r}}\label{subs:proof}
We proceed in a number of steps.\\
{\it Step 1.} We consider a sequence of auxiliary problems in strictly convex domains. Let $\Omega_n$ be a region bounded by four properly chosen circle arcs passing trough vertices of $\Omega$ and such that the Hausdorff distance between $ \Omega_n$  and $\Omega$ is less than $1/n$. 
In other words,  
$$
\d\Omega_n = \d\Omega + \nu\gamma_n,
$$
where $\nu$ is the outer normal to $\d\Omega$ (except for the corners), $\gamma_n$ is a smooth function (away from the corners) with $0\le \gamma_n\le 1/n$.
We define functions $f_n$ on $\partial \Omega_n$ with
$
f_n(x + \nu\gamma_n(x)):=  f(x),
$
 $x\in \partial \Omega$.
By \cite{sternberg}  the following problem, where $T_ {\d\Omega_n}:BV(\Omega_n) \to L^1(\d\Omega_n)$ is the trace operator, 
\begin{equation}\label{e:pn}
 \min \left\{ \int_{\Omega_n} |D v|: v\in BV(\Omega_n),\, T_{\d\Omega_n}v = f_n\right\}
\end{equation}
has a unique continuous solution $v_n$.

{\it Step 2.} We set $u_n = v_n \chi_\Omega$. Since $f_n$ is continuous on $\partial \Omega_n$, then by \cite[Theorem 3.33]{Demengel} there exists $F_n\in W^{1,1}(\Omega_n)$ such that 
$F_n|_{\partial \Omega_n}=f_n$ and $\|DF_n\|_{L^{1}(\Omega_n)}\leq C(\Omega_n )\|f_n\|_{L^1(\partial \Omega_n)}.$
In fact, the argument in  \cite[Theorem 3.33]{Demengel} works for Lipschitz domains not only for smooth ones. We arrive at the bounds,
$$
\|DF_n\|_{L^{1}(\Omega_n)}\leq C(\Omega_n )||f_n||_{L^1(\partial \Omega_n)}\leq C(\Omega_n)||f_n||_{L^\infty(\partial \Omega_n)}|\partial \Omega_n|\leq C_0(\Omega)||f_n||_{L^\infty(\partial \Omega_n)}= C_0(\Omega) ||f||_{L^\infty(\Omega)}.
$$
with $C_0$ constant independent of $n$.
By the definition of $u_n$ and $v_n$, we get that
$$
\int_{\Omega} |Du_n|\leq\int_{\Omega_n} |Dv_n|\leq ||DF_n||_{L^{1}(\Omega_n)} \leq C_0(\Omega) ||f||_{L^\infty(\Omega)}.<\infty. 
$$
Thus, there exists a subsequence, denoted by  $u_n $, converging to $u$ in $L^1(\Omega)$ and
$$
\liminf_{n\to\infty} \int_\Omega  |Du_n| \ge \int_\Omega |Du|.
$$

{\it Step 3.} We prove that $u$ is a function of least gradient with a given trace. Let $g_n$ be the trace of $u_n$ on $\partial \Omega$. Since $v_n$ is continuous on $\bar \Omega_n$, and $u_n=v_n\chi_\Omega$, then the trace is defined as follow:
$$
g_n(z)=\lim_{y\to z, y\in\Omega} u_n(y)=v_n(z).
$$ 
By Proposition \ref{thm:Miranda}, it is enough to show that for every $n$, $u_n$ is a function of least gradient over $\Omega$. For this purpose we will use the following observation.

\begin{proposition}
Let $\Omega_1 \subset \Omega_2$ be open sets with Lipschitz boundary. Then, if $u \in BV(\Omega_2)$ is of least gradient in $\Omega_2$, then $u|_{\Omega_1}$ is of least gradient in $\Omega_1$.
\end{proposition}
\proof Suppose that $u|_{\Omega_1}$ is not a function of least gradient in $\Omega_1$. Explicitly, there exists $v \in BV(\Omega_1)$ such that $T_{\partial \Omega_1} u = T_{\partial \Omega_1} v$ and $|Dv|(\Omega_1) < |Du|(\Omega_1)$. Let us define $\widetilde{u}$ by the formula

\begin{equation*}
\widetilde{u}(x) = \left\{
\begin{array}{ll}
 v(x) & \hbox{if }x \in \Omega_1,\\
 u(x)& \hbox{if }x \in \Omega_2\setminus \Omega_1 .
\end{array}
\right.
\end{equation*}
From \cite[Section 5.4]{evans} it follows that the total variation of $\widetilde{u}$ equals

\begin{equation*}
|D\widetilde{u}|(\Omega_2) = |Dv|(\Omega_1) + |Du|(\Omega_2\setminus \Omega_1) + \int_{\partial \Omega_1} 0 < |Du|(\Omega_1) + |Du|(\Omega_2\setminus \Omega_1) + \int_{\partial \Omega_1} 0 = |Du|(\Omega_2),
\end{equation*} 
which contradicts the fact that $u$ is of least gradient in $\Omega_2$. \qed

We apply this proposition with $\Omega_1 = \Omega$, $\Omega_2 = \Omega_n$ and $u$ equal to $v_n$, while $u_n$ is the restriction of $v_n$ to $\Omega$.
We conclude that $u_n$ are minimizers of the following problem
$$
\left\{ \int_{\Omega_n} |Dv|:v\in BV(\Omega), T v =g_n\right\}
$$
and $u$ is a least gradient function. \\

{\it Step 4.}  We show now that $u_n$ converges uniformly to a continuous function $w$.
Due to Step 3, $w=u$ a.e. and $w$ is a least gradient function. 

We have that $\partial \Omega=\Gamma_1\cup \Gamma_2$, with $\Gamma_1=h_1\cup v_1$, and $\Gamma_2=h_2\cup v_2$ and we are given $f$ continuous and strictly monotone on $\Gamma_i$, $i=1,2$.
%
%
%

For every $t \in (\min f, \max f)$, we denote by $\ell^t$ the line segment with endpoints  $x^t$, $y^t$ such that $x^t\in \Gamma_1$, $y^t\in \Gamma_2$, and $f(x^t)=f(y^t)=t$. Notice that by the monotonicity of $f$, segments $\ell^t$ are disjoint. The lemma below explains that the line segment,  $\ell^t$,  fill out region $\Omega$.

\begin{lemma}\label{lem:segment}
For every $z\in \Omega$ there exists a unique $\ell^t$ for $t\in (\min f, \max f)$ such that $z\in \ell^t$.
\end{lemma}
\proof
Without loss of generality, we may assume
that $z$ is closer to $\Gamma_2$, i.e below the diagonal joining the endpoints of $\Gamma_1$. Let $s$ be the arclength parameter of $\Gamma_1$, its range is $(0,2(L+h))$, and let $x(s)$ be the parametrization of $\Gamma_1$ such that
$$
\lim_{s\to 0^+}f(x(s)) =\min_{\partial \Omega} f,\qquad
\lim_{s\to (2(L+h))^-}f(x(s)) = \max_{\partial \Omega} f.
$$
For every $s$, we draw the line $\ell (x(s),z)$ passing through $x(s)$ and $z$ and we call $y(s)$ its intersection with $\Gamma_2$. We notice that $y(s)$ is continuous and
$$
\lim_{s\to 0^+}f(x(s)) - f(y(s))=\min_{\partial \Omega}f-f(y(s)) \leq 0,\quad
\lim_{s\to (2(L+h))^-}f(x(s))- f(y(s))=\max_{\partial \Omega}f-f(y(s)) \geq 0.
$$
Due to strict the monotonicity of the function $s\mapsto f(x(s)) - f(y(s))$,
there exists a unique $s_0$ and $t$ such that $\ell^t = [x(s_0),y(s_0)]$, $z\in \ell^t$, and $t = f(x(s_0)) = f(y(s_0))$. 
\qed

The above lemma guarantees that function $w$ on $\Omega$ given by the formula below is well-defined:
$$
w(x)=t\qquad\hbox{if }x\in \ell^t.
$$
\begin{lemma}\label{lem:continuity}
 $w$ is continuous in $\Omega$.
 \end{lemma}
 
\proof Let us call by $\omega$ the continuity modulus of $f$. Let $x_1, x_2\in \Omega$. For each $i=1,2$ there exists a unique $t_i\in (\min f ,\max f)$ such that $x_i\in \ell^{t_i}=[x^{t_i},y^{t_i}]$. 
There are three cases to consider: 
\begin{enumerate}
\item both $\ell^{t_1}$ and $\ell^{t_2}$ intersect $h_1$ and $h_2$ (or $v_1$ and $v_2$);
\item both $\ell^{t_1}$ and $\ell^{t_2}$ intersect $h_1$ and $v_2$ (or $h_2$ and $v_1$);
\item $\ell^{t_1}$ intersects  $h_1$ and $h_2$ while $\ell^{t_2}$ intersects $v_1$ and $h_2$ (or $\ell^{t_1}$ intersects  $v_1$ and $v_2$, while $\ell^{t_2}$ intersects $v_1$ and $h_2$).
\end{enumerate}
In case (1), we  notice that
$$|x_1-x_2|\geq d(x_1,\ell^t_2)\geq \min\left(d(x^{t_1},\ell^{t_2}),d(y^{t_1},\ell^{t_2})\right)\geq \min\left(|x^{t_1}-x^{t_2}|,|y^{t_1}-y^{t_2}|\right)\sin \alpha, $$
 where $\alpha$ is the smallest angle between the diagonals and the horizontal sides of $\Omega$.
 \begin{equation}\label{eq:case1}
 \min\left(|x^{t_1}-x^{t_2}|,|y^{t_1}-y^{t_2}|\right)\leq \dfrac{|x_1-x_2|}{\sin \alpha}.
 \end{equation}
We may estimate $w(x_1)- w(x_2)$ as follows, assuming that 
$|x^{t_1}-x^{t_2}|=\min \{ |x^{t_1}-x^{t_2}|,|y^{t_1}-y^{t_2}|\}$,
\begin{equation}\label{r:c1}
 |w(x_1)- w(x_2)| = |t_1 -t_2| = | f(x^{t_1})-f(x^{t_2})|\le \omega(|x^{t_1}-x^{t_2}|)
\le \omega(|x_1- x_2|/\sin\alpha).
\end{equation}
In case (2), we may assume without the loss of generality that $d(x^{t_1},(L,-h))\leq d(x^{t_2},(L,-h))$ this implies by the monotonicity of $f$ that $d(y^{t_1},(L,-h))\leq d(y^{t_2},(L,-h))$. Then,
$$|x_1-x_2|\geq d(x_1,\ell^t_2)\geq \min\left(d(x^{t_1},\ell^{t_2}),d(y^{t_1},\ell^{t_2})\right)= \min\left(|x^{t_1}-x^{t_2}|\sin \beta_1,|y^{t_1}-y^{t_2}|\sin \beta_2\right).$$
Here $\beta_1$ is the angle formed by $\ell^{t_2}$ with $h_1$, and $\beta_2$ is the angle formed by $\ell^{t_2}$ with $v_2$.
Notice that
$$
\sin \beta_1=\dfrac{d\left(y^{t_2},(L,-h)\right)}{|y^{t_2}-x^{t_2}|}\geq \dfrac{|y^{t_2}-y^{t_1}|}{\diam(\Omega)},\qquad
\sin \beta_2=\dfrac{d\left(x^{t_2},(L,-h)\right)}{|y^{t_2}-x^{t_2}|}\geq \dfrac{|x^{t_2}-x^{t_1}|}{\diam(\Omega)},
$$
then
$$|y^{t_1}-y^{t_2}||x^{t_1}-x^{t_2}|\leq \diam(\Omega)\,|x_2-x_1|.$$
Therefore,
\begin{equation}\label{eq:case2}
\min\left(|x^{t_1}-x^{t_2}|,|y^{t_1}-y^{t_2}|\right)\leq \sqrt{\diam(\Omega)}\, \sqrt{|x_2-x_1|}
\end{equation}
and we similarly estimate $|w(x_1)- w(x_2)|$ in terms of $\omega$,
\begin{equation}\label{r:c2}
 |w(x_1)- w(x_2)|\le \omega\left(\sqrt{\diam(\Omega)}\, \sqrt{ |x_1- x_2|}\right).
\end{equation}

In case (3), we  have that $y^{t_1}$, $y^{t_2}$ belong to $h_2$. Then,
$y^0 = \frac{1}{2}(y^{t_1} +y^{t_2})$ also belongs to $h_2$. We notice that the segment $\ell((L,h),y^0)$ intersects the line $\ell(x_1,x_2)$ at a point $x^0$.
Furthermore, $|x_1-x_2|\geq|x_1-x^0|$, $|x_2-x^0|$. Estimating $|x_1-x^0|$ using case $1$, we get
$$\dfrac{|x_1-x^0|}{\sin \alpha}\geq \min\left(|x^{t_1}-(L,h)|,|y^{t_1}-y^{0}|\right).$$
Using case (2), we have 
$$
\sqrt{\diam(\Omega)}\sqrt{|x_2-x^0|}
\geq  \min\left(|(L,h)-x^{t_2}|,|y^{0}-y^{t_2}|\right).$$
Combining these estimates with (\ref{r:c1}) and  (\ref{r:c2}) we obtain,
\begin{equation*}
 |w(x_1)- w(x_2)|\le |w(x_1)- w(x_0)|+  |w(x_0)- w(x_2)|   \le
 \omega(|x_1- x_2|/\sin\alpha) + \omega(\sqrt{\diam(\Omega)}\, \sqrt{ |x_1- x_2|}). 
\end{equation*}


We conclude that $w$ is continuous, in fact for any $x_1$, $x_2\in\Omega$,
\begin{equation}\label{e:gw}
 |w(x_1) -w(x_2)| =| t_1 -t_2 |= |f(p_1) - f(p_2)| \le 
\omega\left(c_1|x_1-x_2|+c_2\sqrt{|x_1-x_2|}\right)
\end{equation}
where $c_1$, and $c_2$ depend only on $\Omega$.
\qed

\begin{lemma}\label{lem:uniform}
$u_n$ converges uniformly to $w$.
\end{lemma}

\proof We evaluate
$|u_l(x) - u_k(x)|$ for $x\in \Omega$. Let us set, $u_l(x) = t_1$, $u_k(x) = t_2$. We notice that
$x \in \d \{u_k \ge t_1\} = \ell^{t_1}_k$ and $x \in \d \{u_l \ge t_2\} = \ell^{t_2}_l$. By construction,
the endpoints of $\ell^{t_1}_k$ and $\ell^{t_1}$ and respectively of $\ell^{t_2}_l$ and $\ell^{t_2}$ are at the distance not bigger than $\frac{1}{n}$, $n=\min\{k,l\}$.
Hence, we can find $x_k\in \ell^{t_1}$ (resp. $x_l\in \ell^{t_2}$) such that
$d(x, x_k)\le 1/k$, (resp. $d(x, x_l)\le 1/l)$.

Now, by \eqref{e:gw}
$$
|u_l(x) - u_k(x)| =| w(x_k) -  w(x_l)| \le
\omega\left(c_1|x_k-x_l|+c_2\sqrt{|x_k-x_l|}\right)\le  \omega\left(\dfrac{2c_1}{\min\{k,l\}}+\dfrac{c_2\sqrt{2}}{\sqrt{\min\{k,l\}}}\right). 
$$ 

So, $u_n$ converges uniformly to $w$.
\qed

{\it Step 5.} It remains to show that $w$ has the correct trace $f$.

\proof
We may assume without the loss of generality that $z\in \Gamma_1$. Take a sequence $x_n$ converging to $z$. Let $\ell^{t_n}
=[x^{t_n},y^{t_n}]$ be the corresponding segments such that $f(x^{t_n})=f(y^{t_n})=t_n$ and let $\ell^t=[z,z']$ 
be such that $f(z)=f(z')$.
As in Step 4, we have that,
$$
|w(x_n)-f(z)|=|t_n-t|\leq \omega\left(c_1|x_n-z|+c_2\sqrt{|x_n-z|}\right)
$$
Letting $n\to \infty$, we conclude that $\lim_{x\to z, x\in \Omega}w(x)=f(z)$. 

{\it Step 6.} The uniqueness of solutions follows from the structure of superlevel sets of any solution to \eqref{rr}. In fact, one can prove that if $v$ is a solution to $\eqref{rr}$ then $\partial\{v\geq t\}\cap \partial \Omega\subseteq f^{-1}(t)$, and hence by minimality of the superlevel sets, we conclude that $\partial\{v\geq t\}=\partial \{w\geq t\}$ and hence by the same argument presented in Lemma \eqref{p:8}, uniqueness follows.
 \qed

\subsection{Back to FMD}\label{sec:bFMD}
We will use the method used in the proof of Theorem \ref{t:r}.
We assume that in problem \eqref{eq:FMD}, $\Omega = (-L,L)\times(-h,h)$ and $g$ is a constant load $l_T$ acting on $T=[-t,t]$ and $l_B$ acting on $B=[-b,b]$ satisfying $\int_{\d\Omega} g\,dx =0$. In Mechanics, in this situation we say that the load is 
self-equilibrated, see \cite{duvaut}. 
We notice that since $g=\frac{\d f}{\d \tau}$, then
$$
f (x,y) = \left\{
\begin{array}{ll}
0 & y\in [-h,h], x=-L;\\
 l_B \min\{ (x+b)^+, 2b\}& x\in [-L,L], y= -h;\\
 l_T \min\{ (x+t)^+, 2t\}&  x\in [-L,L], y= h;\\
2 b l_B , & y\in [-h,h], x=L.
\end{array}
\right.
$$

\begin{proposition}\label{p:9}
 Let us suppose that  $\Omega = (-L,L)\times(-h,h)$ and $f$ is given by the formula above. Then, there exists a unique solution to (\ref{rr}). 
\end{proposition}
\proof Notice that when $t=L$ and $b=h$, then $f$ is constant on $h_1$, $h_2$ and strictly monotone on $v_1$ and $v_2$ and a similar argument to the proof of Theorem \ref{t:r} can be used to show the existence of solutions to this case. Otherwise, notice that $f$ has the property, which was the basis of the construction in the proof of Theorem \ref{t:r}. Namely, each value of $f$, except for the maximum and minimum  is attained at exactly two points.
Let us define $k_\ep:\d\Omega\to \bR$ by the following formula,
$$
k_\ep(x, -h) = -  (x-b)^+\frac{\ep}{L-b} - (x+ b)^-\frac{\ep}{L-b}, \quad
k_\ep(x, h) = -  (x-t)^+\frac{\ep}{L-t} - (x+ t)^-\frac{\ep}{L-t},
$$
$$
k_\ep(L,y) = - l_Bb - \ep,\quad k_\ep(-L,y) = - \ep.
$$
We apply Theorem \ref{t:r} in the case $f_\ep = f+ k_\ep$ the resulting solution is $u_\ep$. We notice that $f- f_\ep$ is monotone on 
$\{-h\} \times [-b,b]$ and $\{h\} \times [-t,t]$. Let us denote by $Q$ the quadrilateral with vertices
$(-b,-h)$, $(b,-h)$, $(-t,h)$, $(t,h)$, then we conclude that 
$$
0\le |u_\ep - u| \le \ep.
$$
Then, due to Proposition \ref{thm:Miranda} our claim follows. \qed

We notice that the structure  of a solution to (\ref{eq:partialLG}) is simple: $\d\{ u \ge t\}$ are intervals connecting appropriate points on $T$ and $B$. Moreover, the load $g$, defining $f$, need not be piecewise constant.
%

\section{An explicit example: a strictly convex domain, piecewise  constant data}\label{s:e}
We present an example of a solution to (\ref{eq:partialLG}) and  (\ref{eq:partialFMD}), when $\Om$ is strictly convex with smooth boundary. We notice that Definition \ref{def:normaltrace} permits $g$ to be a distribution e.g. 
\begin{equation}\label{ex1}
 g = \sum_{i=1}^3 c_i\delta_{a_i},
\end{equation}
where $\delta_p$ notes a delta function located at $p\in \Omega$. Data given by this formula are admissible as long as $\int_{\d\Omega}\,d g= 0$, i.e. $\sum_{i=1}^3 c_i =0$. 

We assume that $\d\Om$ is parametrized by arclength, $[0,L)\ni s\mapsto x(s)\in \d\Om$. In this case
$g$ of the form \eqref{ex1} may be written as
\begin{equation}\label{r:f}
g = (\alpha_1+\alpha_2) \delta_{x(s_1)} - \alpha_2 \delta_{x(s_2)} -\alpha_1\delta_{x(0)}.
\end{equation}
Since $g=\frac{\d f}{\d \tau}$,
then $f$ is given by the following formula, 
\begin{equation}\label{ex2}
f=\left\{
\begin{array}{ll}
 0, & s\in [0,s_1),\\
\alpha_1+\alpha_2, & s\in [s_1,s_2),\\
\alpha_1, & s\in [s_2,L).
\end{array}
\right.
\end{equation}
We may assume that $\alpha_1$ and $\alpha_2$ are positive and in fact (\ref{ex2}) is the only possible datum up to the exact value and location on $s_1, s_2, L$. Other admissible functions $f$ are obtained from (\ref{ex2}) by shift in $s$ or change of orientation of the curve.

If in the problem (\ref{eq:LG}) data $f$ are given by (\ref{ex2}), then \cite{mazon} guarantees the existence but not uniqueness of solutions. We can show existence differently, by approximating $f$ with continuous data and passing to the limit. 
\begin{proposition}
 Let us suppose that $\Omega$ is bounded and strictly convex and $f$ is given by (\ref{ex2}). Then, there exists a unique solution to 
 \begin{equation}\label{r:gl}
 \min\left\{ \int_\Omega |D u|: u\in BV(\Omega), \, T
 u= f\right\}.
\end{equation}
 and whose superlevel sets are given by formula (\ref{r:det}) below.
\end{proposition}
\proof Let us write $x_i = x(s_i)$, $i=0,1,2$, where $x(\cdot)$ parametrization used above. 
Let us take a sequence $f_\ep$ of continuous functions on $\d\Omega$ such that $f=f_\ep$ on 
$\{x\in\d\Omega: \dist(x,x_i)\ge \ep,\ i=0,1,2\}$ and $u^\ep$ a corresponding sequence of solutions to
\begin{equation}\label{r:gw}
 \min\left\{ \int_\Omega |D u^\e| :u^\ep\in BV(\Omega),\ T
 u = f_\ep\right\}.
\end{equation}
We can see that the structure of $\d\{u^\ep\ge f_\ep\}$ does not change if the intersection
$\d\{u^\ep\ge f_\ep\}\cap \d\Omega$ is at a distance greater than $\ep$ from the points $x_0$, $x_1$, $x_2$.
Thus,
$$
\| u^\ep - u^\de\|_{L^1(\Omega)} \le C\max\{\ep, \de\}(|x_0-x_1| + |x_0-x_2| + |x_2-x_1| ).
$$
As a result, $u^\ep \to u$ in $L^1$. Moreover, we can see that $Tu = f$. By the Stability Theorem, see Proposition  \ref{thm:Miranda}, 
$u$ is a function of least gradient satisfying our boundary conditions, hence $u$ is a solution to (\ref{r:gl}).
\qed

We notice that we have three distinguished points on $\d\Om$, they are
$$
x_0 = x(0), \quad x_1 = x(s_1), \quad x_2 = x(s_2).
$$
We have three candidates for $\d \{ u\ge \alpha_1 + \alpha_2\}$, if $u$ is a solution to  (\ref{ri1}),
$$
[x_0,x_1], \quad [x_1,x_2] , \quad [x_0,x_2].
$$
However, some choices are excluded.
\begin{corollary}
 If $u$ is a solution to (\ref{r:gl}), then it is not possible that $\d \{ u = t\}$ is any triangle.
\end{corollary}
\proof Let us suppose the contrary, i.e. there is $t\in\bR$ such that $\d \{ u = t\}$ is a triangle $\bigtriangleup def$, then by \cite{mazon} we have
$$
0 = \int_{\bigtriangleup def} \di \frac{\nabla u}{|\nabla u|} = \int_{\bigtriangleup def} \di z
= \int_{\d\bigtriangleup def} z\cdot \nu.
$$
Here, $z$ is a vector field with values in $L^\infty$, $\|z\|_\infty\le 1$ and agreeing with $\frac{\nabla u}{|\nabla u|}$ on the sides of $\bigtriangleup def$, see  \cite{mazon} for more details.

Of course on $\d\bigtriangleup def$, $z$ has to be equal to $\pm\nu$, where $\nu$ is the outer normal 
to $\d\bigtriangleup def$. Thus
$$
 \int_{\d\bigtriangleup def} z\cdot \nu = \pm |de| \pm |ef| \pm |ef|.
$$
The last sum may never vanish due to the triangle inequality. \qed

Therefore, we conclude that
\begin{equation}\label{r:det}
 \d \mathcal E_t =\left\{
\begin{array}{ll}
\ \emptyset & t\le 0\hbox{ or } t>\alpha_1+\alpha_2,\\ 
\ [x_0, x_1] & t\in (0,\alpha_1],\\
\ [x_1,x_2] & t\in (\alpha_1,\alpha_1+\alpha_2],
\end{array}
\right.
\end{equation}
where $\mathcal E_t$ are the superlevel sets of $u$.

We may use the argument developed in the course of proof of Theorem \ref{main2} to claim uniqueness despite the lack of continuity.

\section*{Acknowledgement} The work of PR was in part supported by  the Research Grant no 2013/11/B/ST8/04436 financed by the National
Science Centre, entitled: Topology optimization of engineering structures. An approach synthesizing the methods of:
free material design, composite design and Michell-like trusses. 

PR also thanks Professor Tomasz Lewi\'nski of Warsaw Technological University for stimulating discussions.

\end{document}